\documentclass[a4paper,11pt,oneside]{article}

\usepackage{a4wide}
\usepackage[utf8]{inputenc}
\usepackage[T1]{fontenc}
\usepackage{textcomp}
\usepackage[english]{babel}
\usepackage{amsmath}
\usepackage{amssymb}
\usepackage{graphicx}
\usepackage{amsthm}
\usepackage{hyperref}
\usepackage{enumitem}
\usepackage{csquotes}
\usepackage{mathrsfs}
\usepackage{color}
\usepackage{bm}
\usepackage{mathtools}

\usepackage{pgfplots}
\pgfplotsset{compat=1.18}

\DeclareMathOperator{\R}{\mathbb{R}}

\DeclareMathOperator{\N}{\mathbb{N}}

\DeclareMathOperator{\Z}{\mathbb{Z}}
\DeclareMathOperator{\leb}{\mathcal{L}}

\DeclareMathOperator{\prob}{\mathbb{P}}

\newcommand{\coint}[1]{\left[{#1}\right)} 
\newcommand{\ccint}[1]{\left[{#1}\right]} 
\newcommand{\ooint}[1]{\left({#1}\right)} 
\newcommand{\ocint}[1]{\left({#1}\right]}

\newcommand{\modu}[1]{\left\vert{#1}\right\vert}

\addto\captionsfrench{}

\theoremstyle{definition}
\newtheorem{defi}{Definition}[section]
\newtheorem{ex}[defi]{Example}
\newtheorem{rem}[defi]{Remark}

\theoremstyle{plain}
\newtheorem{prop}[defi]{Proposition}
\newtheorem{thm}[defi]{Theorem}
\newtheorem{lemma}[defi]{Lemma}
\newtheorem{cor}[defi]{Corollary}

\title{The $p$-spectrum of Random Wavelet Series}
\author{C\'eline Esser$^a$\footnote{Corresponding author. Email:
    celine.esser@uliege.be} , Thelma Lambert$^a$ and B\'eatrice Vedel$^b$\\\\
$^a$ Universit\'e de Li\`ege,  All\'ee de la D\'ecouverte 12, B-4000
Li\`ege, Belgium \\
$^b$ Universit\'e Bretagne Sud, CNRS UMR 6205, LMBA, F-56000 Vannes, France} 
\date{}

\begin{document}
\maketitle

\begin{abstract}
The goal of multifractal analysis is to characterize variations in local regularity of functions by computing the Hausdorff dimension of sets of points sharing the same regularity. While classical approaches rely on H\"older exponents and are restricted to locally bounded functions, $p$-exponents extend this framework to functions locally in $L^p$ and allow one to describe negative regularities. 
We establish a wavelet-based upper bound for the $p$-spectrum in terms of the asymptotic distribution of wavelet coefficients, extending the classical H\"older case. We then compute the exact $p$-spectrum of \textit{Random Wavelet Series} and show that, for non-negative regularities, this bound is sharp and is also attained by a prevalent set of functions with a prescribed wavelet statistic. Finally, we show that a different phenomenon occurs for negative regularities: contrary to the classical H\"older case, Random Wavelet Series do not in general realize the maximal $p$-spectrum compatible with a prescribed distribution of wavelet coefficients.
\end{abstract}

\noindent  {\bf Keywords :} Multifractal Analysis, Multifractal Formalism, Random wavelet series,  Large deviation spectrum, $p$-exponent\\

\noindent {\bf 2010 Mathematics Subject Classification : } 42C40, 28A80, 26A16, 60G17\\

\section{Introduction}

Multifractal analysis provides a framework to describe the fluctuations of pointwise regularity in functions, signals, and sample paths of stochastic processes, see e.g. \cite{ArneodoBacryMuzy1998,Balanca2014,BarralFournierJaffardSeuret2010,BarralSeuret2007,jaffard04,Jaffard1996,Jaffard1997,JaffardMartin2017}. Over the past decades, it has become a standard tool in signal and image processing and has found applications in a wide range of fields, including physics, finance, neuroscience, and urban studies; see, among many others, \cite{MandMemor,AbryWendtJaffard2012,Arneodo:02,Arneodo1999,Gerasimova2013,LAROCCA:2018:A,Lashermes2005,Lavallee1993,Lengyel2022,pexp2,Mandelbrot1974,Parisi:85,WenAbrJaf07,LevyVehel1994}. Traditionally, this analysis has focused on locally bounded functions whose pointwise regularity can be characterized by H\"older exponents. Recall that, for $\alpha>0$ and $x_0 \in \R$, a locally bounded function $f$ belongs to the \textit{H\"older space} $C^\alpha(x_0)$ if there exist a positive constant $C$ and a polynomial $P$ of degree less than $\alpha$ such that
\[
\modu{f(x)-P(x)}\leq C\modu{x-x_0}^\alpha
\]
for every $x$ in a neighbourhood of $x_0$. As $\alpha$ grows, the condition required to belong to $C^\alpha(x_0)$ becomes increasingly restrictive. It is therefore natural to characterize the regularity of $f$ at $x_0$ by determining its \textit{H\"older exponent} defined by
\[
h_f(x_0)=\sup\{\alpha\geq0:f\in C^\alpha(x_0)\}.
\]
Given the possibly erratic behaviour of the function $x_0\mapsto h_f(x_0)$, one usually seeks to determine a geometric interpretation of the different singularities that appear in $f$ and  their significance. The \emph{multifractal} or \textit{singularity spectrum} of $f$ defined by 
\[
\mathscr{D}_f:\ccint{0,+\infty}\rightarrow\{-\infty\}\cup\ccint{0,1}:h\mapsto\dim_{\mathcal{H}}\big\{x_0\in\R:h_f(x_0)=h\big\}
\]
aims to provide such a description. By convention, the Hausdorff dimension of the empty set is equal to $-\infty$, and the support of the spectrum is defined as the set of H\"older exponents actually observed. See Section~\ref{secHaus} for a brief review of the Hausdorff dimension.

As soon as a function satisfies a H\"older-type condition at $x_0$, it is bounded on a neighbourhood of $x_0$, which justifies the study of H\"older exponents being limited to locally bounded functions. However, many functions of interest in both theoretical and applied contexts are not locally bounded, rendering the classical notion of pointwise H\"older regularity meaningless. To overcome this limitation, Calderón and Zygmund introduced in 1961 the concept of $p$-exponents, which generalize the H\"older exponent to functions that are locally in $L^p$ by substituting the $L^\infty_{\text{loc}}$-norm with any $L^p_{\text{loc}}$-norm \cite{cald}.

\begin{defi}
Fix $p\in\coint{1,+\infty}$ and $f\in L^p_{\text{loc}}(\R)$. If $\alpha\geq\frac{-1}{p}$ and $x_0\in\R$, then  $f$ belongs to the space $T^p_\alpha(x_0)$ if there exist a positive constant $C$, a polynomial $P$ of degree less than $\alpha$ and a positive radius $R$ such that for every $r\leq R$, 
\[
\Big(\frac{1}{r}\int_{B(x_0,r)}\modu{f(x)-P(x)}^p\;dx\Big)^{\frac{1}{p}}\leq Cr^\alpha.
\]
The \textit{$p$-exponent} of $f$ at $x_0$ is then defined as
\[
h_f^{(p)}(x_0)=\sup\Big\{\alpha\geq\frac{-1}{p}:f\in T^p_\alpha(x_0)\Big\}.
\]
\end{defi}

The $p$-exponent measures the rate of decay of local $L^p$ norms of the oscillation of the function around a point and thus provides a natural tool for multifractal analysis in the non-locally bounded setting. The corresponding \textit{$p$-spectrum} describes the size of the sets of points where the $p$-exponent takes a given value, extending the classical multifractal framework.

\begin{defi}
    The \emph{$p$-spectrum} of $f \in L^p_{\text{loc}}(\R)$ is the mapping defined by
\[
\mathscr{D}^{(p)}_f:\Big[\frac{-1}{p},+\infty\Big]\rightarrow\{-\infty\}\cup\ccint{0,1}:h\mapsto\dim_{\mathcal{H}}\big\{x_0\in\R:h^{(p)}_f(x_0)=h\big\}.
\]
\end{defi}

First introduced in the setting of partial differential equations, the concept of $p$-exponents only began to be applied in signal processing much later, once their wavelet-based characterization had been established \cite{jaffard_melot05}. In particular, the studies \cite{pexp1,pexp2} investigate the information on the local behaviour of functions near singularities that can be derived from the collection of $p$-exponents. For additional results concerning $p$-exponents, see \cite{MR3775457,MR3305158,MR3254609,MR2172009,MR3666574}.

For the multifractal analysis of  signals, wavelet methods are among the most powerful and widely used tools available. A function $f \in L^2$ can be expanded in an orthonormal wavelet basis $\psi_{j,k}$, constructed by dilations and translations of a mother wavelet $\psi$. The corresponding wavelet coefficients encode detailed information about the local regularity of the function. By examining their distribution across scales, one can derive sharp estimates of the singularity spectrum and establish a rigorous \emph{multifractal formalism}, that is, a numerically robust framework for estimating the multifractal spectrum. This wavelet-based approach was initially motivated by the study of fully developed turbulence, and has since become a standard methodology for the analysis of complex natural signals \cite{Abry:14,Arneodo:02,Parisi:85}. Since we are interested in local notions, we may from now on consider one-periodic functions and restrict their study to the unit interval. Therefore, we assume that a periodized wavelet basis, indexed by the dyadic tree, is fixed in the Schwartz class. See Section~\ref{secWavelet} for further details on wavelets.

In the present study, we address the problem of estimating the $p$-spectrum from the distribution of wavelet coefficients across scales. As a starting point, we recall the estimates on the singularity spectrum obtained in the classical case $p=+\infty$. To this end, we introduce the notion of wavelet density and wavelet profile: A wavelet coefficients sequence refers to any complex sequence \mbox{$\vec{c}=(c_{j,k})_{j\in\N_0,\,k\in\{0,\ldots,2^j-1\}}$}. To any such sequence $\vec{c}$, and for any $\alpha\in\R$, we associate quantities $\rho_{\vec{c}}(\alpha)$ and $\nu_{\vec{c}}(\alpha)$ such that, intuitively, at each large scale $j$, there are approximately $2^{\rho_{\vec{c}}(\alpha)j}$ coefficients of order $2^{-\alpha j}$ and $2^{\nu_{\vec{c}}(\alpha)j}$ coefficients larger than $2^{-\alpha j}$. These notions are formalized as follows.

\begin{defi}\label{def:prof}
Let $\vec{c}\,$ be a wavelet coefficients sequence. The \textit{wavelet density} and the \textit{wavelet profile} of the sequence $\vec{c}$ are the functions $\rho_{\vec{c}}$ and $\nu_{\vec{c}}$ respectively defined for every $\alpha\in\R$ by
\[
\rho_{\vec{c}}(\alpha)=\lim_{\varepsilon\rightarrow0^+}\limsup_{j\rightarrow+\infty}\frac{1}{j}\log_2\left(\#\{k\in\{0,\ldots,2^j-1\}:2^{-(\alpha+\varepsilon)j}\leq\modu{c_{j,k}}\leq 2^{-(\alpha-\varepsilon)j}\}\right)
\]
and
\[
\nu_{\vec{c}}(\alpha)=\lim_{\varepsilon\rightarrow0^+}\limsup_{j\rightarrow+\infty}\frac{1}{j}\log_2\left(\#\{k\in\{0,\ldots,2^j-1\}:\modu{c_{j,k}}\geq 2^{-(\alpha+\varepsilon)j}\}\right).
\]
Notice that, as soon as $\{\alpha\in\R:\nu_{\vec{c}}(\alpha)=-\infty\}\neq\emptyset$, $\nu_{\vec{c}}$ is the increasing hull of $\rho_{\vec{c}}$, that is,
\begin{equation}\label{eq:enveloppe}
\nu_{\vec{c}}(\alpha)=\sup_{\alpha'\leq \alpha} \rho_{\vec{c}}(\alpha')\quad\forall\alpha\in\R
\end{equation}
(which can be proved as in \cite{largeDev}).
\end{defi}

These quantities play a key role in the upper bound of the multifractal spectrum, as obtained in \cite{RWS}: 
If $f$ is a uniformly H\"older function and if $\vec{c}$ denotes its sequence of wavelet coefficients in a given wavelet basis, then for every $h\geq 0$,
\begin{equation}
\label{eq:bornesupHolder}
\mathscr{D}_f(h)\leq h\sup_{\alpha\in\ocint{0,h}}\frac{\rho_{\vec{c}}(\alpha)}{\alpha} = h \sup_{\alpha\in\ocint{0,h}}\frac{\nu_{\vec{c}}(\alpha)}{\alpha}
\end{equation}
(where the equality follows from Equation~\eqref{eq:enveloppe}).

Furthermore, it was proved in \cite{RWS} that this upper bound \eqref{eq:bornesupHolder} becomes an equality as soon as the wavelet coefficients are independently sampled at each scale according to a fixed distribution, such series being called \textit{Random Wavelet Series}. See Section~\ref{secRWS} for a precise definition of these series.

In addition, it was established in \cite{aubryPrevalence} that, within the so-called $S^{\nu}$ class of functions sharing a prescribed wavelet statistic, the maximal multifractal richness allowed by the distribution of wavelet coefficients across scales is achieved for ``almost all'' functions. 
More precisely, this upper bound is achieved on a \emph{prevalent} subset of the corresponding $S^{\nu}$ space, where prevalence is understood in the sense of Hunt, Sauer, and Yorke \cite{hunt}.  Prevalence provides a natural notion of genericity in infinite-dimensional spaces. We refer to Section~\ref{secSnu} for further details on $S^\nu$ spaces and on the notion of prevalence.

These three properties -- namely, upper bounds that are sharp for Random Wavelet Series and, more generally, for generic functions in certain function spaces -- are crucial to define the right-hand side of $\eqref{eq:bornesupHolder}$ as a valid formalism. In particular, this expression can be employed numerically to estimate the multifractal spectrum, since it typically coincides with or provides a rigorous upper bound for the true  spectrum.

\medskip

In the context of non-locally bounded functions, previous studies mainly focused on specific models such as \emph{Lacunary Wavelet Series} introduced in \cite{LWS}.  In this model, one fixes $\alpha\in\R$ and $\eta\in(0,1)$, and, at each scale $j$, a wavelet coefficient $c_{j,k}$ takes the value $2^{-\alpha j}$ with probability $2^{(\eta-1)j}$, and vanishes otherwise.
 This construction ensures that, on average, there are approximately $2^{\eta j}$ non-zero coefficients at each scale. The parameter $\eta$ therefore controls the lacunarity of the series, whereas $\alpha$ is directly related to its uniform regularity. The exact determination of the $p$-spectrum of Lacunary Wavelet Series was completed in \cite{pLWS}, paving the way for the study of the $p$-spectrum in a more general setting.

The aim of our paper is therefore to extend the three results mentioned in the H\"older case, offering a practical method to estimate the $p$-spectrum from the distribution of wavelet coefficients. To obtain an analogue of Inequality~\eqref{eq:bornesupHolder}, the requirement of being locally in $L^p$ is replaced by a stronger assumption that can be easily read on wavelet coefficients. This assumption relies on the \textit{scaling function} $\eta_f$, which is defined for every $p>0$ by
\[
\eta_f(p)=\liminf_{j\rightarrow+\infty}\frac{-1}{j} \log_2\Big(2^{-j}\displaystyle\sum_{k=0}^{2^j-1}\modu{c_{j,k}}^p\Big),
\]
and more precisely on the best value of $p$ for which the scaling function is positive, i.e.
\begin{equation}\label{eq:p_0}
p_0(f)=\sup\{p>0:\eta_f(p)>0\}.
\end{equation}
The relevance of this condition follows from its connection with Sobolev spaces \cite{pLWS,Jaffard1997}, which yields for $p \geq 1$ the following criterion for local $p$-integrability: if $\eta_f(p)>0$, then $f\in L^p_{\mathrm{loc}}$, whereas if $\eta_f(p)<0$, then $f\notin L^p_{\mathrm{loc}}$ \cite{pexp1}. Moreover, when $\eta_f(p)>0$, the pointwise $p$-regularity can be characterized in terms of wavelet coefficients, see Proposition~\ref{caract_exp}. For $p\in(0,1)$, the pointwise $p$-exponent is understood directly through this wavelet characterization, which remains meaningful under the condition $\eta_f(p)>0$. Thus, the notion of $p$-spectrum naturally extends to values of $p$ smaller than $1$.

%The relevance of this quantity is justified by the following precise criterion for local $p$-integrability: for $p\geq1$, if $\eta_f(p)>0$, then $f\in L^p_{\text{loc}}$, and if $\eta_f(p)<0$, then $f\notin L^p_{\text{loc}}$ \cite{pexp1}. In addition, it allows one to consider values of $p$ in $\ooint{0,1}$.

Our first main result is the following. Notice that the H\"older case corresponds to $p=+\infty$. Formally setting $1/p=0$ in the formula below, we recover the upper bound obtained by Aubry and Jaffard, recalled in Equation~\eqref{eq:bornesupHolder}.

\begin{thm}\label{mainSup}
If $f$ is a function for which $p_0(f)>0$, then, for every $0<p<p_0(f)$ and every $h\geq -\frac{1}{p}$,
\[
    \mathscr{D}_f^{(p)}(h)\leq
    \min\left(
    \max\left(
    \sup_{\alpha\in\coint{h,0}}
    \frac{(\rho_{\vec{c}}(\alpha)-1)h}{\alpha}+1,
    \left(h+\frac{1}{p}\right)
    \sup_{\alpha\in\ocint{-\frac{1}{p},h}}
    \frac{\rho_{\vec{c}}(\alpha)}{\alpha+\frac{1}{p}}
    \right),1\right),
\]
with the convention that $\sup \emptyset=-\infty$. In particular, for non-negative values of $h$, one has 
\[
  \mathscr{D}_f^{(p)}(h)\leq \min\left(
\left(h+\frac{1}{p}\right)
\sup_{\alpha\in\ocint{\frac{-1}{p},h}}
\frac{\rho_{\vec{c}}(\alpha)}{\alpha+\frac{1}{p}},
1
\right).
\]
\end{thm}

Note that, from the definition of $\eta_f(p)$, one directly obtains that
\begin{equation}\label{eq:bound_rho}
    \rho_{\vec{c}}(\alpha) \leq \alpha p+1 -\eta_f(p)
\qquad \text{for every }\alpha\in\R,
\end{equation}
see \cite{RWS}. A similar estimate for the $p$-leader profile will be established in Lemma~\ref{sup_B}.
Hence, for every $h<0$ and every $\alpha\in\coint{h,0}$, one has
\[
\frac{(\rho_{\vec{c}}(\alpha)-1)h}{\alpha}+1
< \frac{\alpha ph}{\alpha}+1 = hp+1 < 1.
\]
Theorem~\ref{mainSup} suggests a natural candidate for a multifractal formalism on the interval $\coint{0,+\infty}$, with a possible extension to $\coint{\frac{-1}{p},+\infty}$. Moreover, it is natural to consider the candidate for the almost-everywhere regularity of $f$ suggested by this upper bound, namely the value of $h$ for which the upper bound reaches $1$. We denote this critical value by $h_{\max}^{(p)}$. This leads to the following definition.

\begin{defi}
Let $f\in L^p_{\mathrm{loc}}$ be a function whose wavelet coefficients form the sequence $\vec{c}$. We define
\[
D^{(p)}_f(h) = \left(h+\frac{1}{p}\right)\sup_{\alpha\in\ocint{\frac{-1}{p},h}}\frac{\rho_{\vec{c}}(\alpha)}{\alpha+\frac{1}{p}},
\]
and set
\[
h^{(p)}_{\max}
=
\inf\left\{h\geq 0 : D_f^{(p)}(h)=1\right\}.
\]
We say that $f$ satisfies the \emph{$p$-large deviation wavelet formalism} if
\[
\mathscr{D}_f^{(p)} = D^{(p)}_f \quad \text{on} \quad \ccint{\frac{-1}{p},h^{(p)}_{\max}}.
\]
\end{defi}

 This choice is consistent with the classical large deviation wavelet formalism and provides a natural extension to finite $p$. As we shall see, however, this extension need not be optimal for negative regularities.

Note that the equality in Equation~\eqref{eq:enveloppe} implies
\[
D_f^{(p)}(h)=\left(h+\tfrac{1}{p}\right)\sup_{\alpha\in\ocint{\frac{-1}{p},h}}\frac{\nu_{\vec{c}}(\alpha)}{\alpha+\frac{1}{p}},
\]
which shows that the $p$-large deviation wavelet formalism can equivalently be defined in terms of the wavelet profile of the sequence of wavelet coefficients.  

Our second main result establishes that it is possible to construct a large class of random functions for which the $p$-large deviation wavelet formalism holds. These functions, called Random Wavelet Series, are defined by choosing the wavelet coefficients at each scale $j$ as independent and identically distributed random variables. Given the probability distribution of the wavelet coefficients at each scale, the wavelet density and the wavelet profile of the resulting Random Wavelet Series are almost surely deterministic and can be expressed in terms of these distributions. We denote their theoretical counterparts by $\bm{\rho}$ and $\bm{\nu}$, respectively; their precise definitions are given in Section~\ref{secRWStot}.

These random series coincide with the processes considered in the classical case $p=+\infty$ in \cite{RWS}, except that here the definition is extended to allow functions that are only locally in $L^p$, rather than necessarily locally bounded.

The parameter $h_{\min}$ involved in the next result is roughly defined through
\[
h_{\min}=\inf\left\{\alpha\in\R:\forall\varepsilon>0,\sum_{j\in\N_0}\mathbb{E}\left[\#\left\{0\leq k<2^j:2^{-(\alpha+\varepsilon)j}\leq\modu{c_{j,k}}\leq2^{-(\alpha-\varepsilon)j}\right\}\right]=+\infty\right\}
\]
Thus, $h_{\min}$ is the smallest asymptotic order of wavelet coefficients that occurs sufficiently often across scales. It corresponds to the left endpoint of the effective support of the wavelet density (and profile), see Section~\ref{secRWStot} for the precise framework and further details.

%The parameter $h_{\min}$ involved in the next result is also defined in Section~\ref{secRWStot}: it is the smallest exponent at which the wavelet profile (or density) takes a finite value.

\begin{thm}\label{main_RWS}
Let $f$ be a Random Wavelet Series.
Then, almost surely, for all $0<p<\inf_{\alpha<0}\frac{\bm{\rho}(\alpha)-1}{\alpha}$, the support of $\mathscr{D}^{(p)}_f$ is $\ccint{h_{\min},h_{\max}^{(p)}}$, and $f$ satisfies the $p$-large deviation wavelet formalism.
\end{thm}

The almost sure $p$-spectrum of a Random Wavelet Series is illustrated in Figure \ref{figureRWS}.

\begin{figure}\label{figureRWS}
    \centering
\begin{tikzpicture}
\begin{axis}[
    width=12cm, height=8cm,
    axis lines=middle,
    legend style={at={(0.5,1.05)},anchor=south},
    xtick=\empty, 
    ytick=\empty,  
    xmin=-0.65, xmax=1.2, 
    ymin=-0.15, ymax=1.5, 
    x=8cm,  
    y=3.8cm
]

\addplot[blue, thick, smooth, domain=-0.07:1.02] {x + 0.2*sin(10*deg(x)) + 0.2};

\addplot[gray, dashed, domain=-0.6:0.043] {0.5*(0.6+ x)};
\addplot[gray, dashed, domain=-0.6:-0.03] {0.2*(0.6+x)};
\addplot[gray, dashed, domain=-0.6:0.6] {0.735*(0.6+x)};
\addplot[gray, dashed, domain=-0.6:0.8] {0.855*(0.6+x)};

\addplot[red, thick, domain=-0.07:0.149] {x + 0.2*sin(10*deg(x)) + 0.2};
\addplot[red, thick, domain=0.148:0.66] {0.735*(0.6+x)};
\addplot[red, thick, domain=0.66:0.8] {x + 0.2*sin(10*deg(x)) + 0.2};
\addplot[red, thick, domain=0.8:0.93] {0.855*(0.6+x)};

  \draw (-0.6,-0.01)--(-0.6,0.01);
  \draw (-0.6,0) node[below]  {\footnotesize $-\frac{1}{p}$};
  
  \draw (-0.07,-0.01)--(-0.07,0.01);
  \draw (-0.07,0) node[below]  {\footnotesize $h_{\min}$};

  \draw[loosely dashed] (-0.01,1.31)--(0.93,1.31);
  \draw (-0.01,1.31) node[left]  {\footnotesize $1$};

  \draw[loosely dashed] (0.93,-0.01)--(0.93,1.31);
  \draw (0.93,0) node[below]  {\footnotesize $h^{(p)}_{\max}$};
  \draw (0.93,1.31) node {\tiny $\bullet$};

\end{axis}
\end{tikzpicture} 

\caption{The almost sure $p$-spectrum of a Random Wavelet Series (in red) together with the corresponding wavelet density (in blue). }
\end{figure}

As in the classical case $p=+\infty$, one can show that, when an asymptotic distribution of wavelet coefficients is prescribed, the $p$-large deviation wavelet formalism is satisfied by a prevalent set of functions on $\coint{0,+\infty}$. This constitutes our third main result. Here, the coefficient distribution is described by a so-called admissible profile $\nu$, which defines the space $S^\nu$. In this setting, functions are allowed to be locally in $L^p$, rather than necessarily locally bounded. In Section~\ref{secSnu}, we provide precise definitions of admissible profiles and of the quantity $p_\nu$ used below, which guarantees that $\eta_f(p)>0$ for every $f\in S^\nu$ and every $p<p_\nu$. We also clarify the role of $h_{\min}$ and $h_{\max}^{(p)}$, analogous to those introduced in the context of Random Wavelet Series, and show that these quantities can be determined solely from the profile $\nu$.

\begin{thm}\label{mainPrev}
For a prevalent set of functions $f$ in $S^\nu$, for all $0<p<p_\nu$,
\[
\operatorname{supp}\mathscr{D}_f^{(p)}\cap\coint{0,+\infty}
=
\ccint{\max(0,h_{\min}),h_{\max}^{(p)}},
\]
and $f$ satisfies the $p$-large deviation wavelet formalism on $\coint{0,+\infty}$.
\end{thm}

A striking difference with the classical case $p=+\infty$ appears for negative regularities. In the classical setting, Random Wavelet Series realize the maximal multifractal spectrum compatible with a prescribed asymptotic distribution of wavelet coefficients. For finite $p$, this is no longer true in general: the spatial arrangement of the coefficients may create additional singularities that are not captured by their asymptotic distribution alone.

More precisely, clusters of wavelet coefficients of order $\alpha$ may generate $p$-leaders of a lower order $h<\alpha$, and hence produce a $p$-spectrum larger than the one predicted by the $p$-large deviation wavelet formalism. This phenomenon is ruled out for Random Wavelet Series by the spatial uniformity induced by the independent sampling of their coefficients, but it may occur in a more general setting. The following proposition shows that such clustering effects can indeed destroy the formalism for negative regularities.

\begin{prop}\label{prop:CounterEx}
Fix $p>0$ and $h,\alpha\in\R$ satisfying $\frac{-1}{p}<h<\alpha<0$. Then there exists a function $f$ such that $\eta_f(p)>0$ and
\[
\mathscr{D}_f^{(p)}(h)
=
\frac{(\rho_{\vec{c}}(\alpha)-1)h}{\alpha}+1
>
\Big(h+\frac{1}{p}\Big)
\sup_{\alpha'\in\ocint{\frac{-1}{p},h}}
\frac{\rho_{\vec{c}}(\alpha')}{\alpha'+\frac{1}{p}}
.
\]
\end{prop}

The function $f$ will moreover be constructed so that
$\rho_{\vec{c}}(\alpha')=-\infty$ for every $\alpha'\neq\alpha$. In particular, the right-hand side of the $p$-large deviation wavelet formalism does not recover the value of the $p$-spectrum at $h$. Thus, contrary to the classical H\"older case, Random Wavelet Series do not in general realize the maximal $p$-spectrum compatible with a prescribed asymptotic distribution of wavelet coefficients at negative regularities. The spatial organization of the coefficients therefore plays an essential role.

\medskip

Our paper is organized as follows. In Section~\ref{secNot}, we introduce wavelets, recall the necessary notation, and define the local $\ell^p$-norms of wavelet coefficients ($p$-leaders), which allow us to characterize pointwise $p$-regularity. We also review Hausdorff measure and dimension. Section~\ref{secSup} is devoted to the proof of Theorem~\ref{mainSup}. In Section~\ref{secRWStot}, we focus on Random Wavelet Series, provide a precise definition of these processes, and prove Theorem~\ref{main_RWS}. In this section, we refine the upper bound established in Theorem~\ref{mainSup} and derive a matching lower bound for the spectrum, thereby proving that the formalism holds. We also compute the almost sure $p$-scaling function of the Random Wavelet Series.  In Section~\ref{secSnutot}, we recall the notion of prevalence and the spaces $S^\nu$, and prove Theorem~\ref{mainPrev}. Finally, in Section~\ref{secCounterEx}, we prove Proposition~\ref{prop:CounterEx} by constructing functions whose wavelet coefficients all have a fixed order $\alpha$, but whose $p$-spectrum contains points of regularity $h<\alpha$. 

Throughout the paper, $\mathbb{N}$ denotes the set $\{1,2,\ldots\}$ of positive integers, whereas $\mathbb{N}_0$ denotes the set $\{0,1,2,\ldots\}$ of non-negative integers. Moreover, $\lceil \cdot \rceil$ denotes the ceiling function, defined for every $x\in\mathbb{R}$ by $\lceil x \rceil = \min\{n\in\mathbb{Z}:x\leq n\}$. We also adopt the conventions
$\sup\emptyset=-\infty$, $\inf\emptyset=+\infty$, and $\frac{1}{+\infty}=0.$
Finally, we restrict ourselves to dimension $1$ in order to keep the notation and the proofs as transparent as possible. We expect that the same approach can be adapted to higher dimensions, but we do not pursue this extension here.

\section{Notations and definitions}\label{secNot}

\subsection{Wavelets and leaders}\label{secWavelet}

We consider a mother wavelet $\psi$ in the Schwartz class.\footnote{A compactly supported wavelet could be used as well, provided that its regularity is larger than the maximal pointwise regularity of the signal.} Then the collection
\[
\left\{2^{\frac{j}{2}}\psi_{j,k} :j\in\N,k\in\{0,\ldots,2^j-1\}\right\}\cup\{\psi_{0,0}=1\},
\]
where $\psi_{j,k}$ is the periodized wavelet $\psi_{j,k}(x)=\sum_{l\in\Z}\psi\left(2^j(x-l)-k\right)$, $x\in\ccint{0,1}$, forms an orthonormal basis of $L^2(\ccint{0,1})$ (see \cite{LM86,Daubechies:92}). We use a $L^\infty$-normalisation, in which case any one-periodic function $f \in  L^2$ can be written as
\[
f=\sum_{j\in\N_0}\sum_{k=0}^{2^j-1}c_{j,k}\psi_{j,k},
\]
where the wavelet coefficients of $f$ are defined by
\[
c_{j,k}=2^j\int_0^1\psi_{j,k}(x)f(x)\,dx.
\]
Note that the wavelet coefficients can be defined even when $f$ does not belong to $L^2$.

Dyadic intervals are commonly used to index wavelets and wavelet coefficients: if we set \mbox{$\lambda_{j,k}=\coint{k2^{-j},(k+1)2^{-j}}$}, then $\psi_{\lambda_{j,k}}=\psi_{j,k}$ and $c_{\lambda_{j,k}}=c_{j,k}$ for every $k\in\{0,\ldots,2^j-1\}$ and every $j\in\N_0$. Therefore, for every $j\in\mathbb N_0$, we identify the set of dyadic intervals
\[
\Lambda_j=\{\lambda_{j,k}:k=0,\ldots,2^j-1\}
\]
with the corresponding set of indices $\{0,\ldots,2^j-1\}$. We also denote by $\Lambda$ both the collection of all dyadic intervals in $[0,1]$ and the set of pairs $(j,k)$ with $j\in\mathbb N_0$ and $k\in\Lambda_j$.
With these notations, for a fixed wavelet basis, the function $f$ is identified with a sequence $\vec{c}$ of $\R^\Lambda$. Finally, in the context of pointwise properties, it is useful to refer to $\lambda_j(x_0)$ as the only dyadic interval of $\Lambda_j$ that contains $x_0$.

\bigskip

One can investigate the pointwise regularity of a function $f$ through its wavelet coefficients $\vec c$. In the Hölder case, the \textit{wavelet leaders}
\begin{equation}\label{eq:leader}
l_\lambda=\sup_{j'\geq j}\sup_{\lambda'\in\Lambda_{j'},\,\lambda'\subseteq 3\lambda}\modu{c_{\lambda'}}\quad (\lambda\in\Lambda_j,j\in\N_0),
\end{equation}
allow one to characterize the Hölder exponent \cite{jaffard04}. Their $p$-analogues, called \textit{$p$-leaders}, similarly provide a characterization of $p$-exponents. In this work, we use the modified definition introduced in \cite{laurent}, in which the local supremum over coefficients at each finer scale is replaced by a weighted $\ell^p$-norm.

\begin{defi}
Fix $p>0$, a scale $j\in\N_0$ and a dyadic interval $\lambda\in\Lambda_j$. The \textit{$p$-leader} associated with $\lambda$ is 
\[
l_\lambda^{(p)}=\sup_{j'\geq j}\Big(\sum_{\lambda'\in\Lambda_{j'},\,\lambda'\subseteq 3\lambda}\modu{c_{\lambda'}}^p2^{-(j'-j)}\Big)^{\frac{1}{p}}.
\]
\end{defi}

The main purpose of $p$-leaders is to provide the following characterization of pointwise $p$-regularity. Note that for $p\in(0,1)$, the right-hand side below is used to define the $p$-exponent, as in \cite{pLWS}.

\begin{prop}\label{caract_exp}\cite{pexp1,laurent}
Let $f:\ccint{0,1}\rightarrow\R$ and $p\geq 1$ be such that $\eta_f(p)>0$. Then for every $x_0\in\ccint{0,1}$, 
\[
h^{(p)}_f(x_0)
=\liminf_{j\rightarrow+\infty}\frac{\log\big(l_{\lambda_j(x_0)}^{(p)}\big)}{\log \left(2^{-j}\right)}.
\]
\end{prop}

\subsection{Hausdorff measure and dimension}\label{secHaus}

A few fundamental concepts are recalled in this section; for a more complete treatment, see, e.g. \cite{Falconer:86}. Let $A$ be a subset of $\R^d$ and $\xi:\coint{0,+\infty}\rightarrow\coint{0,+\infty}$ be a function such that $\xi(0)=0$ and $\xi$ is increasing on a neighbourhood of 0. The \textit{Hausdorff outer measure} at scale $t\in\ocint{0,+\infty}$ associated with $\xi$ of the set $A$ is defined by
\[
\mathcal{H}_t^\xi(A)=\inf\left\{\sum_{n\in\N}\xi(\text{diam}(E_n)):\text{diam}(E_n)\leq t \text{ and } A\subseteq \bigcup_{n\in\N}E_n\right\}.
\]
The \textit{Hausdorff measure} associated with $\xi$ of the set $A$ is defined by $\mathcal{H}^\xi(A)=\lim_{t\rightarrow0^+}\mathcal{H}^\xi_t(A)$.
If $\xi(x)= x^s$ with $s\geq 0$, one simply uses the usual notations $\mathcal{H}^\xi_t(A)=\mathcal{H}^s_t(A)$ and \mbox{$\mathcal{H}^\xi(A)=\mathcal{H}^s(A)$}, and these measures are called \textit{$s$-dimensional Hausdorff outer measure} at scale $t$ and \textit{$s$-dimensional Hausdorff measure} respectively.

\medskip

If $A$ is non-empty, it can be proved that the function $s\mapsto \mathcal{H}^s(A)$ is non-increasing and that there exists a critical value $h\geq 0$ such that $\mathcal{H}^s(A)=+\infty$ for all $s\in\coint{0,h}$ and $\mathcal{H}^s(A)=0$ for all $s\in\ooint{h,+\infty}$. This threshold value $h$ is called the \textit{Hausdorff dimension} of $A$. More precisely,
\[
\dim_{\mathcal{H}}A=\left\{\begin{array}{ll}
\inf\{s\geq 0:\mathcal{H}^s(A)=0\} & \text{if }A\neq\emptyset, \\
-\infty & \text{if }A=\emptyset.
\end{array}\right.
\]
Moreover, if there exists a gauge function $\xi$ such that
\[
\lim_{r\rightarrow0^+}\frac{\log\xi(r)}{\log r}=h
\]
and $\mathcal{H}^\xi(A)>0$, then $\dim_{\mathcal{H}}(A)\geq h$.

\section{Upper bound for the $p$-spectrum}\label{secSup}

The aim of this section is to prove Theorem~\ref{mainSup}, that is, to provide an upper bound for the $p$-multifractal spectrum of any fixed function $f$, with $p_0(f)>0$, and any fixed $p<p_0(f)$, recalling that $p_0(f)$ is defined in Equation~\eqref{eq:p_0}. This upper bound is obtained using large deviation estimates on the distribution of the wavelet coefficients $\vec{c}$ of $f$.

\subsection{Large deviation estimates of $p$-leaders}\label{secSup1}

We can define \textit{$p$-leaders versions} of the wavelet density and the wavelet profile.

\begin{defi}
Let $(l^{(p)}_\lambda)_{\lambda\in\Lambda}$ be the $p$-leaders sequence associated with $\vec{c}$. 
The \textit{$p$-leaders density} and the \textit{$p$-leaders profile} of $\vec{c}$ are the functions $\rho^{(p)}_{\vec{c}}$ and $\nu^{(p)}_{\vec{c}}$ respectively defined by
\[
\rho^{(p)}_{\vec{c}}(\alpha)=\lim_{\varepsilon\rightarrow0^+}\limsup_{j\rightarrow+\infty}\frac{1}{j}\log_2\left(\#\{\lambda\in\Lambda_j:2^{-(\alpha+\varepsilon)j}\leq l^{(p)}_\lambda\leq 2^{-(\alpha-\varepsilon)j}\}\right)
\]
and
\[
\nu^{(p)}_{\vec{c}}(\alpha)=\lim_{\varepsilon\rightarrow0^+}\limsup_{j\rightarrow+\infty}\frac1j\log_2\left(\#\{\lambda\in\Lambda_j:l^{(p)}_\lambda\geq 2^{-(\alpha+\varepsilon)j}\}\right),
\]
for every $\alpha\in\R$.
\end{defi}

To avoid the overlap in the sums defining two neighboring $p$-leaders -- which is important to preserve independence across dyadic intervals at the same scale when working with independent random wavelet coefficients -- we use the following restricted definition and correspondingly adapt the definitions of the density and profile.

\begin{defi}
Fix a scale $j\in\N_0$, and a dyadic interval $\lambda\in\Lambda_j$. The \textit{restricted $p$-leader} associated with $\lambda$ is defined by
\[
e_\lambda^{(p)}=\sup_{j'\geq j}\Big(\sum_{\lambda'\in\Lambda_{j'},\,\lambda'\subseteq \lambda}\modu{c_{\lambda'}}^p\,2^{-(j'-j)}\Big)^{\frac{1}{p}}.
\]
If we consider these restricted $p$-leaders instead of the classical ones in the definitions of the $p$-leaders density and profile, we denote the resulting functions by $\rho_{\vec{c}}^{(p),\ast}$ and $\nu_{\vec{c}}^{(p),\ast}$, in place of $\rho_{\vec{c}}^{(p)}$ and $\nu_{\vec{c}}^{(p)}$.
\end{defi}

Let us now compare the density based either on restricted or non-restricted modified $p$-leaders.

\begin{prop}\label{class_restr}
For every $\alpha\in\R$, one has
$
\rho_{\vec{c}}^{(p)}(\alpha)\leq\rho_{\vec{c}}^{(p),\ast}(\alpha).
$
\end{prop}

\begin{proof}
It follows from the fact that for every large scale $j$ and every $\lambda\in\Lambda_j$, if $N(\lambda)$ denotes the set of the three neighbours of $\lambda$ in $\Lambda_j$, then
\[
l_\lambda^{(p)}=\Big(\sum_{\mu\in N(\lambda)}\big(e_\mu^{(p)}\big)^p\Big)^{\frac{1}{p}},
\]
which entails that for every $\varepsilon>0$, 
\[
\#\{\lambda\in\Lambda_j:2^{-(\alpha+\varepsilon)j}\leq 
l_\lambda^{(p)}\leq2^{-(\alpha-\varepsilon)j}\} \leq 3\cdot\#\{\lambda\in\Lambda_j:2^{-(\alpha+2\varepsilon)j}\leq e_\lambda^{(p)}\leq2^{-(\alpha-2\varepsilon)j}\}.
\]
\end{proof}

Note that, in the case of the $p$-leaders profile, the functions $\nu_{\vec{c}}^{(p)}$ and $\nu_{\vec{c}}^{(p),\ast}$  actually coincide on $\R$. This result can be obtained as in \cite{largeDev}, where the classical case  $p=+\infty$ is treated.

\subsection{Proof of Theorem~\ref{mainSup}}\label{secSup2}

The proof of  Theorem~\ref{mainSup} is decomposed into Proposition~\ref{spectre_densite}, the previously established Proposition~\ref{class_restr} and Theorem~\ref{petitsoleil}, each of which proves one of the following inequalities:
\[
\mathscr{D}_f^{(p)}(h)\leq 
\rho_{\vec{c}}^{(p)}(h)\leq 
\rho_{\vec{c}}^{(p),\ast}(h)\leq\max\Bigg(\sup_{\alpha\in\coint{h,0}}\frac{(\rho_{\vec{c}}(\alpha)-1)h}{\alpha}+1,\left(h+\frac{1}{p}\right)\sup_{\alpha\in\big(\frac{-1}{p},h\big]}\frac{\rho_{\vec{c}}(\alpha)}{\alpha+\frac{1}{p}}\Bigg).
\]

The proof of Proposition~\ref{spectre_densite} works verbatim as in \cite{largeDev}, where the results are established in the case $p=+\infty$. It relies on Lemma \ref{lm_spectre_densite}, which itself follows immediately from the characterization of $p$-exponents via $p$-leaders.

\begin{lemma}\label{lm_spectre_densite}
For every $\alpha\geq -1/p$, define
\[
F_j^{(p)}(\alpha)=\big\{\lambda\in\Lambda_j:l_\lambda^{(p)}\geq 2^{-\alpha j}\big\} \quad\text{and}\quad E^{(p)}(\alpha)=\limsup_{j\rightarrow+\infty}\bigcup_{\lambda\in F_j^{(p)}(\alpha)}\lambda.
 \]
Then
\[
\big\{x:h_f^{(p)}(x)<\alpha\big\}
\subseteq
E^{(p)}(\alpha)
\subseteq
\big\{x:h_f^{(p)}(x)\leq\alpha\big\}.
\]
\end{lemma}

\begin{prop}\label{spectre_densite}
For every $h\geq\frac{-1}{p}$, we have 
$\mathscr{D}^{(p)}_f(h)\leq 
\rho_{\vec{c}}^{(p)}(h).$
\end{prop}

The central part of this section is therefore to bound the large-deviation estimates of restricted $p$-leaders by our formalism. We will need the following Lemma, which enhances a result of \cite{pexp2} stating that for every $h\geq\frac{-1}{p}$,
\[
\mathscr{D}^{(p)}_f(h)\leq hp+1.
\]

\begin{lemma}\label{sup_B}
For every $h\geq-\frac1p$, one has
\[
\nu_{\vec{c}}^{(p),*}(h)
\leq hp+1-\eta_f(p).
\]
\end{lemma}
\begin{proof}
Fix $h\geq\frac{-1}{p}$ and $\delta<\eta_f(p)$. By definition of $\eta_f$, there exists $J\in\N$ such that for every $j'\geq J$,
\[
2^{-j'}\sum_{\lambda'\in\Lambda_{j'}}\modu{c_{\lambda'}}^p<2^{-\delta j'}.
\]
It follows that for every $j\geq J$ and every $\varepsilon>0$, one has 
\[
\#\big\{\lambda\in\Lambda_j:e_\lambda^{(p)}\geq 2^{-(h+\varepsilon)j}\big\} \leq 2^{(h+\varepsilon)pj}\sum_{\lambda\in\Lambda_j}\Big(\sup_{j'\geq j}\;\sum_{\lambda'\subseteq\lambda}\modu{c_{\lambda'}}^p2^{-(j'-j)}\Big)<2^{(hp+1+\varepsilon p-\delta)j}.
\]
Hence, for every $\delta<\eta_f(p)$,
\[
\nu_{\vec{c}}^{(p),*}(h)\leq hp+1-\delta.
\]
Letting $\delta$ tend to $\eta_f(p)$ yields the conclusion.
\end{proof}

\begin{thm}\label{petitsoleil}
For every $h\geq\frac{-1}{p}$, we have
\[
\rho_{\vec{c}}^{(p),\ast}(h)\leq\max\Big(\sup_{\alpha\in\coint{h,0}}\left(\frac{(\rho_{\vec{c}}(\alpha)-1)h}{\alpha}+1\right),\left(h+\frac{1}{p}\right)\sup_{\alpha\in\big(\frac{-1}{p},h\big]}\frac{\rho_{\vec{c}}(\alpha)}{\alpha+\frac{1}{p}}\Big).
\]
\end{thm}

We decompose the proof of Theorem~\ref{petitsoleil} into several lemmas. We first distinguish the case $h<h_{\min}$ from the case $h\geq h_{\min}$, where
\begin{equation}\label{eq:hmin}
    h_{\min}=\frac{\eta_f(p)}{p}-\frac{1}{p}.
\end{equation}
We begin with the case $h<h_{\min}$.

\begin{lemma}\label{lm:small_h}
For every $h<h_{\min}$, we have
$
\rho_{\vec{c}}^{(p),\ast}(h)=-\infty.
$
\end{lemma}

\begin{proof}
For every $h<h_{\min}$ and every $\varepsilon>0$ such that $h+\varepsilon<h_{\min}$, since $p\big(h+\varepsilon+\frac{1}{p}\big)<\eta_f(p)$, there exists $J\in\N$ such that for all $j'\geq J$, 
\[
2^{-j'}\sum_{\lambda'\in\Lambda_{j'}}\modu{c_{\lambda'}}^p<2^{-\big(h+\varepsilon+\frac{1}{p}\big)pj'}.
\]
Therefore for all $j\geq J$ and all $\lambda\in\Lambda_j$,
\[
e_\lambda^{(p)}\leq2^{\frac{j}{p}}\sup_{j'\geq j}2^{-\big(h+\varepsilon+\frac{1}{p}\big)j'}=2^{-(h+\varepsilon)j}.
\]
The conclusion follows. 
\end{proof}

% This case being settled, we fix $h\geq h_{\min}$, we assume 
% \begin{equation}\label{eq:hyp}
%     \rho^{(p),\ast}_{\vec{c}}(h)>\sup_{\alpha\in\coint{h,0}}\left(\frac{(\rho_{\vec{c}}(\alpha)-1)h}{\alpha}+1\right),
% \end{equation}
% we define $M\in\N$ arbitrarily if $h\geq 0$ and such that $h<\frac{-2}{Mp}$ if $h<0$,
% we consider $\varepsilon>0$ satisfying
% \begin{equation}\label{eq:epsilon}
%     \rho_{\vec{c}}^{(p),\ast}(h)>4p\varepsilon \quad\text{if }\rho_{\vec{c}}^{(p),\ast}(h)>0,
% \end{equation}
% \begin{equation}\label{eq:epsilon_hyp}
%     \rho^{(p),\ast}_{\vec{c}}(h)>\sup_{\alpha\in\coint{h,0}}\left(\frac{(\rho_{\vec{c}}(\alpha)-1)h}{\alpha}+1\right)+M\varepsilon p 
% \end{equation}
% and
% \begin{equation}\label{eq:epsilon_neg}
%     h+2\varepsilon ,\; M\varepsilon p +\frac{Mhp}{2}+1<0 \quad\text{if }h<0, 
% \end{equation}
% and we prove
% \begin{equation}\label{eq:thesis_sup}
%     \rho_{\vec{c}}^{(p),\ast}(h)\leq\left(h+\frac{1}{p}\right)\sup_{\alpha\in\ocint{\frac{-1}{p},h}}\frac{\rho_{\vec{c}}(\alpha)}{\alpha+\frac{1}{p}}.
% \end{equation}
This case being settled, we now fix $h\geq h_{\min}$ and assume that
\begin{equation}\label{eq:hyp}
    \rho^{(p),\ast}_{\vec{c}}(h)>
    \sup_{\alpha\in\coint{h,0}}
    \Big(    \frac{(\rho_{\vec{c}}(\alpha)-1)h}{\alpha}+1 \Big).
\end{equation}
We also choose $M\in\N$ arbitrarily if $h\geq 0$, and such that $h<\frac{-2}{Mp}$ if $h<0$.
Next, we fix $\varepsilon>0$ small enough so that
\begin{equation}\label{eq:epsilon}
    \rho_{\vec{c}}^{(p),\ast}(h)>4p\varepsilon \quad\text{if }\rho_{\vec{c}}^{(p),\ast}(h)>0,
\end{equation}
\begin{equation}\label{eq:epsilon_hyp}
    \rho^{(p),\ast}_{\vec{c}}(h)>\sup_{\alpha\in\coint{h,0}}\Big(\frac{(\rho_{\vec{c}}(\alpha)-1)h}{\alpha}+1\Big)+M\varepsilon p 
\end{equation}
and
\begin{equation}\label{eq:epsilon_neg}
    h+2\varepsilon ,\; M\varepsilon p +\frac{Mhp}{2}+1<0 \quad\text{if }h<0, 
\end{equation}
Under these assumptions, we will prove that
\begin{equation}\label{eq:thesis_sup}
    \rho_{\vec{c}}^{(p),\ast}(h)
    \leq
    \Big(h+\frac{1}{p}\Big)
    \sup_{\alpha\in\big(\frac{-1}{p},h\big]}
    \frac{\rho_{\vec{c}}(\alpha)}
    {\alpha+\frac{1}{p}}.
\end{equation}
% For every $j \in \mathbb{N}$, we are interested in the set of dyadic intervals 
% \(\Lambda_j^{(p)}(h,\varepsilon)\) defined by
% \[
% \Lambda_j^{(p)}(h,\varepsilon) := \bigl\{ \lambda \in \Lambda_j \ : \ 2^{-(h+\varepsilon)j} \leq e_\lambda^{(p)} \leq 2^{-(h-\varepsilon)j} \bigr\}.
% \]
% For every such interval $\lambda \in \Lambda_j^{(p)}(h,\varepsilon)$, in order to derive from the relation
% $$
% 2^{-(h+\varepsilon)j} \leq e_\lambda^{(p)} \leq 2^{-(h-\varepsilon)j}
% $$
%  a control over the wavelet coefficients, we need to determine the "dominating behaviour" of $e_\lambda^{(p)}$, i.e. to find a scale $j'(\lambda)$ and an order $\alpha(\lambda)$ such that
% \[
% e_\lambda^{(p)}\sim \left(\sum_{\lambda'\in\Lambda_{j'(\lambda)},\,\lambda'\subseteq\lambda,\,\modu{c_{\lambda'}}\sim 2^{-\alpha(\lambda)j'(\lambda)}}\modu{c_{\lambda'}}^p2^{-(j'(\lambda)-j)}\right)^{\frac{1}{p}}.
% \]
For every $j \in \mathbb{N}$, we consider the set of dyadic intervals
\[
\Lambda_j^{(p)}(h,\varepsilon):=\big\{\lambda \in \Lambda_j :2^{-(h+\varepsilon)j}\leq e_\lambda^{(p)}\leq 2^{-(h-\varepsilon)j} \big\}.
\]
For every $\lambda \in \Lambda_j^{(p)}(h,\varepsilon)$, we aim to derive information on the wavelet coefficients from the estimate
$$
2^{-(h+\varepsilon)j}\leq e_\lambda^{(p)}\leq 2^{-(h-\varepsilon)j}.
$$
To this end, we first identify a scale $j'(\lambda)\geq j$ at which the restricted $p$-leader is essentially attained, and then determine an order $\alpha(\lambda)$ of the coefficients that dominates the corresponding sum. In other words, we seek $j'(\lambda)$ and $\alpha(\lambda)$ such that
\[
e_\lambda^{(p)}\sim \Bigg(\sum_{\substack{\lambda'\in\Lambda_{j'(\lambda)},\,\lambda'\subseteq\lambda\\\modu{c_{\lambda'}}\sim 2^{-\alpha(\lambda)j'(\lambda)}}}
\modu{c_{\lambda'}}^p2^{-(j'(\lambda)-j)}\Bigg)^{\frac{1}{p}}.
\]
We begin by locating such a scale $j'(\lambda)$ and showing that it is bounded by a constant multiple of $j$.
To that end, we fix any exponent $\alpha_0 \in (\frac{-1}{p},h_{\min})$ and a scale $J\in\N$ such that
\begin{equation}\label{sum_max}
2^{-j'}\sum_{\lambda'\in\Lambda_{j'}}\modu{c_{\lambda'}}^p<2^{-\big(\alpha_0+\frac{1}{p}\big)pj'} \quad \forall j' \geq J. 
\end{equation}
Recall that we are now considering the case $h\geq h_{\min}$. Since $\alpha_0<h_{\min}$, it follows in particular that $h>\alpha_0$.

\begin{lemma}\label{lm:scale}
For every $j\geq J$ and every $\lambda\in\Lambda_j^{(p)}(h,\varepsilon)$, there exists $j'(\lambda)\geq j$ such that
\begin{equation}\label{sum_order}
2^{-\big(h+\frac{3\varepsilon}{2}\big)pj}\leq\sum_{\lambda'\subseteq\lambda}\modu{c_{\lambda'}}^p2^{-(j'(\lambda)-j)}\leq 2^{-(h-\varepsilon)pj}
\end{equation}
and
\begin{equation}\label{scale_max}
j'(\lambda)\leq \frac{h+2\varepsilon+\frac{1}{p}}{\alpha_0+\frac{1}{p}} j.
\end{equation}
\end{lemma}

\begin{proof}
For every $j\geq J$ and every $\lambda\in\Lambda_j^{(p)}(h,\varepsilon)$,
there exists $j'(\lambda)\geq j$ such that
\[
2^{-\big(h+\frac{3\varepsilon}{2}\big)j}\leq \Big(\sum_{\lambda'\subseteq\lambda}\modu{c_{\lambda'}}^p2^{-(j'(\lambda)-j)}\Big)^{\frac{1}{p}}\leq 2^{-(h-\varepsilon)j},
\]
hence \eqref{sum_order} and
\begin{equation}\label{sum_min}
2^{-j'(\lambda)}\sum_{\lambda'\subseteq\lambda}\modu{c_{\lambda'}}^p
\geq2^{-\big(h+2\varepsilon+\frac{1}{p}\big)pj}.
\end{equation}
Inequality~\eqref{sum_max} applied to $j'=j'(\lambda)$ and Inequality~\eqref{sum_min} directly imply
Condition~\eqref{scale_max}.
\end{proof}

We now discretize the possible values of $j'(\lambda)$ by considering multiples of $j$ of the form $A(\lambda)j$, where $A(\lambda)$ belongs to a finite set $\mathcal{A}$ independent of $j$. To this end, fix $m\in\N$ sufficiently large so that the conditions \eqref{eq:m_neg} and \eqref{eq:m} below are satisfied; in particular, $m$ depends only on $h$, $\varepsilon$, $\alpha_0$ and $p$. We then define
\[
\mathcal{A}=\Bigg\{a+\frac{b}{m}:a\in\Big\{1,\ldots,\frac{N}{m}\Big\}, b\in\{0,\ldots,m-1\}\Bigg\},
\]
where $N\in\N$ is chosen such that
$$\frac{N}{m}+1=\left\lceil\frac{h+2\varepsilon+\frac{1}{p}}{\alpha_0+\frac{1}{p}}\right\rceil.$$
With this notation, the following result is an immediate consequence of Lemma~\ref{lm:scale}.

\begin{cor}\label{cor:scale}
To any $j\geq J$ and any $\lambda\in\Lambda_j^{(p)}(h,\varepsilon)$, we can associate $A(\lambda)\in\mathcal{A}$ such that Equation~\eqref{sum_order} is satisfied for 
\[
j'(\lambda)\in\Big[A(\lambda)j,\Big(A(\lambda)+\frac{1}{m}\Big)j\Big].
\]
\end{cor}

We now seek a sequence of scales $(j_n)_{n\in\N}$ and a value $A\in\mathcal{A}$ such that, for many intervals $\lambda\in\Lambda_{j_n}^{(p)}(h,\varepsilon)$, the associated scale $j_n'(\lambda)$ is independent of $\lambda$ and remains close to $Aj_n$.
If $\rho_{\vec{c}}^{(p),\ast}(h)>0$, we fix $\delta>0$ sufficiently small so that
\begin{equation}\label{eq:delta}
    0<\delta<\min\left(\frac{1}{5m},\frac{1}{5}\,\rho_{\vec{c}}^{(p),\ast}(h),\eta_f(p),\frac{\rho^{(p),\ast}_{\vec{c}}(h)-\displaystyle\sup_{\alpha\in\coint{h,0}}\left(\frac{(\rho_{\vec{c}}(\alpha)-1)h}{\alpha}+1\right)-M\varepsilon p}{4+\frac{h\left(h+2\varepsilon+\frac{1}{p}\right)}{\left(\alpha_0+\frac{1}{p}\right)(h+2\varepsilon)}}\right) 
\end{equation}
Such a choice is possible in view of \eqref{eq:epsilon_hyp}. If $\rho_{\vec{c}}^{(p),\ast}(h)=0$, we set $\delta=0$.
By the definition of $\rho_{\vec{c}}^{(p),\ast}(h)$, we can then choose an increasing sequence $(j_n)_{n\in\N}$, with $j_1\geq J$, such that for every $n\in\N$,
\begin{equation}\label{card_Lambdap}
    \#\Lambda_{j_n}^{(p)}(h,\varepsilon)
    \geq
    2^{\left(\rho_{\vec{c}}^{(p),\ast}(h)-\delta\right)j_n}.
\end{equation}
Finally, when $\delta>0$, we increase $J$ if necessary so that $N<2^{\delta J}$.

\begin{lemma}\label{lm:card_Lambdap}
    There exists $A\in\mathcal{A}$ and a sequence $(j_n')_{n\in\N}$ such that for every $n\in\N$ large enough,
    \[
        \#\widetilde{\Lambda}_{j_n}^{(p)}(h,\varepsilon) \coloneqq \#\left\{\lambda\in\Lambda_{j_n}^{(p)}(h,\varepsilon):A(\lambda)=A \text{ and }j_n'(\lambda)=j_n'\right\} \geq 2^{\big(\rho_{\vec{c}}^{(p),\ast}(h)-3\delta\big)j_n}.
    \]
    In particular, for every $n\in\N$, $j_n'\in\ccint{Aj_n,\left(A+\frac{1}{m}\right)j_n}$.
\end{lemma}

\begin{proof}
Clearly, for every $n\in\N$,
\[
\#\Lambda_{j_n}^{(p)}(h,\varepsilon) = \sum_{a=1}^{\frac{N}{m}}\sum_{b=0}^{m-1}\#\Lambda^{(p)}_{j_n}\big(h,\varepsilon,a+\frac{b}{m}\big),
\]
where
 $
\Lambda_{j_n}^{(p)}(h,\varepsilon,A)=\big\{\lambda\in\Lambda_{j_n}^{(p)}(h,\varepsilon):A(\lambda)=A\big \}.
$
Using Equation~\eqref{card_Lambdap}, we know that for every $n\in\N$,
\[
2^{(\rho_{\vec{c}}^{(p),\ast}(h)-\delta)j_n}\leq N\sup_{A\in\mathcal{A}}\#\Lambda_{j_n}^{(p)}(h,\varepsilon,A),
\]
from which follows the existence of $A_n\in\mathcal{A}$ such that
$
2^{(\rho_{\vec{c}}^{(p),\ast}(h)-2\delta)j_n}\leq\#\Lambda_{j_n}^{(p)}(h,\varepsilon,A_n).
$
Using the pigeon hole principle, 
we may assume that there exists $A\in\mathcal{A}$ such that for every $n\in\N$, 
\[
2^{(\rho_{\vec{c}}^{(p),\ast}(h)-2\delta)j_n}\leq \#\Lambda_{j_n}^{(p)}(h,\varepsilon,A).
\]
Moreover, for every $n\in\N$ and every 
$\lambda\in\Lambda_{j_n}^{(p)}(h,\varepsilon,A)$, $j'_n(\lambda)\in\ccint{Aj_n,\left(A+\frac{1}{m}\right)j_n}$. Therefore, for every $n\in\N$ large enough so that $\frac{j_n}{m}+1\leq 2^{\delta j_n}$ if $\delta>0$, one integer value of $\ccint{Aj_n,\left(A+\frac{1}{m}\right)j_n}$ must be picked at least $2^{(\rho_{\vec{c}}^{(p),\ast}(h)-3\delta)j_n}$ times and we call $j_n'$ this value.
\end{proof}

\medskip

Secondly, for every $\lambda\in\widetilde{\Lambda}_{j_n}^{(p)}(h,\varepsilon)$, we need to identify the order of the coefficients that gives the dominant contribution to the sum at scale $j'_n$. More precisely, we seek an exponent $\alpha(\lambda)$ such that
\[
\sum_{\lambda'\in\Lambda_{j'_n},\,\lambda'\subseteq\lambda}\modu{c_{\lambda'}}^p
\]
is essentially governed by the contribution of those coefficients satisfying
$
\modu{c_{\lambda'}}\sim 2^{-\alpha(\lambda)j'_n}.
$
To discretize the possible orders of the coefficients, we fix $\beta>0$ and $L\in\N_0$ such that
\[
\beta<\alpha_0+\frac{1}{p},
\qquad
\beta<\frac{1}{m},
\qquad
\frac{h+2\varepsilon+\frac{1}{p}}{\beta}=L+1.
\]
Moreover, if $h<0$ (recall that in this case $h+2\varepsilon<0$ by \eqref{eq:epsilon_neg}), for every
\[
\alpha\in I\coloneqq
\ccint{
h+2\varepsilon+\frac{1}{p},
\frac{(h+2\varepsilon)\left(\alpha_0+\frac{1}{p}\right)}
{h+2\varepsilon+\frac{1}{p}}
+\frac{1}{p}
},
\]
we choose $\varepsilon(\alpha)>0$ such that
\begin{equation}\label{eq:epsilon_s1}
    \varepsilon(\alpha)
    <
    \frac{\varepsilon}
    {6\left(A+\frac{1}{m}\right)}
\end{equation}
and, for every $j'$ large enough,
\begin{equation}\label{eq:epsilon_s2}
    \#\Big\{
    \lambda'\in\Lambda_{j'}:
    2^{-\big(\alpha-\frac{1}{p}+\varepsilon(\alpha)\big)j'}
    \leq \modu{c_{\lambda'}}
    \leq
    2^{-\big(\alpha-\frac{1}{p}-\varepsilon(\alpha)\big)j'}
    \Big\}
    \leq
    2^{\big(\rho_{\vec{c}}\left(\alpha-\frac{1}{p}\right)+\delta\big)j'}.
\end{equation}
We then choose a finite covering
\[
I\subseteq
\bigcup_{s=1}^S
\ooint{\alpha_s-\varepsilon_s,\alpha_s+\varepsilon_s},
\]
with $\alpha_s\in I$ and $\varepsilon_s=\varepsilon(\alpha_s)$ for every $s\in\{1,\ldots,S\}$.

The following lemma discretizes the possible orders of the coefficients $\modu{c_{\lambda'}}$ with $\lambda'\in\Lambda_{j'_n}$ and $\lambda'\subseteq\lambda$.

\begin{lemma}\label{lm:list_orders}
For every $n\in\N$ large enough, every $\lambda\in\widetilde{\Lambda}_{j_n}^{(p)}(h,\varepsilon)$ and every $\lambda'\in\Lambda_{j'_n}$ with $\lambda'\subseteq\lambda$, either $\modu{c_{\lambda'}}<2^{-(h+2\varepsilon)j_n}$, or there exists $s\in\{1,\ldots,S\}$ such that
\[
2^{-\big(\alpha_s-\frac{1}{p}+\varepsilon_s\big)j'_n} \leq \modu{c_{\lambda'}} \leq 2^{-\big(\alpha_s-\frac{1}{p}-\varepsilon_s\big)j'_n}, 
\]
or there exists $l(\lambda')\in\{1,\ldots,L\}$ such that
\[
2^{-\left(l(\lambda')\beta+\frac{1}{m}\right)j'_n}2^{\frac{j'_n}{p}}\leq\modu{c_{\lambda'}}\leq 2^{-l(\lambda')\beta j'_n}2^{\frac{j'_n}{p}}.
\]
Moreover, the first case cannot happen simultaneously for all the intervals $\lambda'$ considered and the second case can only happen if $h<0$.
\end{lemma}

\begin{proof}
    The last statement follows from the lower bound in Equation~\eqref{sum_order} and the fact that $I$ is empty if $h\geq 0$.

    Fix $n\in\N$ large enough and $\lambda\in\widetilde{\Lambda}_{j_n}^{(p)}(h,\varepsilon)$. From Equation~\eqref{sum_max} follows that for every $\lambda'\in\Lambda_{j_n'}$,
    \[
    \modu{c_{\lambda'}} \leq 2^{-\alpha_0j'_n}.
    \]
Moreover, the lower bound in \eqref{sum_order} shows that the inequality
$$
\modu{c_{\lambda'}}<2^{-(h+2\varepsilon)j_n}
$$
cannot hold for every $\lambda'\in\Lambda_{j'_n}$ with $\lambda'\subseteq\lambda$. Consider now a coefficient which does not satisfy this first alternative.    If $h<0$, recall that $h+2\varepsilon<0$, and either
$$
2^{-(h+2\varepsilon)j_n}\leq \modu{c_{\lambda'}}<2^{-(h+2\varepsilon)j'_n},
\quad \text{ or } \quad
2^{-(h+2\varepsilon)j'_n}\leq \modu{c_{\lambda'}}\leq2^{-\alpha_0j'_n}.
$$
If $h\geq0$, only the latter case can occur.    To any dyadic interval $\lambda'$ of scale $j'_n$ with $\lambda'\subseteq\lambda$, we denote by $\alpha(\lambda')\geq 0$ the exponent such that
    \begin{equation}\label{coeff_order}
    \modu{c_{\lambda'}}=2^{-\alpha(\lambda') j'_n}2^{\frac{j'_n}{p}}.
    \end{equation}
    In the second case, 
    \[
    (h+2\varepsilon)j_n' < \alpha(\lambda')j_n'-\frac{j_n'}{p} \leq (h+2\varepsilon)j_n,
    \]
    i.e.
    \[
    h+2\varepsilon+\frac{1}{p} < \alpha(\lambda') \leq (h+2\varepsilon)\frac{j_n}{j_n'}+\frac{1}{p},
    \]
    hence $\alpha(\lambda')\in I$ in view of Relation~\eqref{scale_max}, which concludes the case. In the third case,
    \[
    \alpha_0+\frac{1}{p} \leq \alpha(\lambda') \leq h+2\varepsilon+\frac{1}{p}
    \]
  Since 
$
    \bigcup_{l=1}^L\ccint{l\beta,(l+1)\beta}
    $
 covers $\big[\alpha_0+\frac{1}{p},h+\frac{1}{p}+2\varepsilon\big]$, there exists $l(\lambda')\in\{1,\ldots,L\}$ such that \eqref{coeff_order} is satisfied with 
    $    \alpha(\lambda')\in\ccint{l(\lambda')\beta,(l(\lambda')+1)\beta}.    $
Since $\beta<1/m$, this implies
$$
\alpha(\lambda')
\in
\ccint{l(\lambda')\beta,l(\lambda')\beta+\frac{1}{m}},
$$
hence the conclusion.
\end{proof}

We now introduce some notation to count the coefficients of order $l\beta-\frac{1}{p}$.

\begin{defi}\label{def:count}
For every $n\in\N$ large enough, every $\lambda\in\widetilde{\Lambda}_{j_n}^{(p)}(h,\varepsilon)$ and every $l\in\{1,\ldots,L\}$, we define $r_\lambda(l)\in\{-\infty\}\cup\ccint{\frac{j_n}{j'_n},1}$ such that
\[
\#\Big\{\lambda'\subseteq\lambda:2^{-\left(l\beta+\frac{1}{m}\right)j'_n}2^{\frac{j'_n}{p}}\leq \modu{c_{\lambda'}}\leq 2^{-l\beta j'_n}2^{\frac{j'_n}{p}}\Big\}=2^{r_\lambda(l)j'_n-j_n}.
\]
We further define $l_0(\lambda)$ as (one of) the value(s) in $\{1,\ldots,L\}$ such that
\[
\sup_{l\in\{1,\ldots,L\}}\left(r_\lambda(l)A-l\beta pA\right)
=r_\lambda(l_0(\lambda))A-l_0(\lambda)\beta pA.
\]
\end{defi}

Lemma~\ref{lm:sum_order} will show that the dominant contribution to the sum corresponds to the order
$
\alpha(\lambda)=l_0(\lambda)\beta-\frac{1}{p}.
$
For this lemma, we assume that $n\in\N$ is large enough so that
$
L+S+1<2^{\frac{\varepsilon}{2}pj_n}.
$
To establish this result, we first show that the sum cannot be dominated by coefficients of order $\alpha_s-\frac{1}{p}$. This is the purpose of the next two lemmas, throughout which we assume that
\begin{equation}\label{eq:m_neg}
    m\geq\frac{-3(h+2\varepsilon)}{\varepsilon}.
\end{equation}

\begin{lemma}\label{lm:rid_order1}
    If $h<0$, then for every $s\in\{1,\ldots,S\}$ such that $A\leq\frac{h+2\varepsilon}{\alpha_s-\frac{1}{p}}$, for every $n\in\N$ large enough and every $\lambda\in\widetilde{\Lambda}_{j_n}^{(p)}(h,\varepsilon)$,
    \[
    \sum_{\lambda'\subseteq\lambda \,:\, 2^{-\left(\alpha_s-\frac{1}{p}+\varepsilon_s\right)j_n'} \leq \modu{c_{\lambda'}} \leq 2^{-\left(\alpha_s-\frac{1}{p}-\varepsilon_s\right)j_n'}} \modu{c_{\lambda'}}^p2^{-(j_n'-j_n)} \leq 2^{-\left(h+\frac{5\varepsilon}{3}\right)pj_n}.
    \]
\end{lemma}

\begin{proof}
    One has
    \begin{align*}
        \sum_{\lambda'\subseteq\lambda \,:\, 2^{-\left(\alpha_s-\frac{1}{p}+\varepsilon_s\right)j_n'} \leq \modu{c_{\lambda'}} \leq 2^{-\left(\alpha_s-\frac{1}{p}-\varepsilon_s\right)j_n'}} \modu{c_{\lambda'}}^p2^{-(j_n'-j_n)}
        & \leq 2^{-\left(\alpha_s-\frac{1}{p}-\varepsilon_s\right)pj_n'} \\
        & \leq 2^{-\left(\alpha_s-\frac{1}{p}-\varepsilon_s\right)p\left(A+\frac{1}{m}\right)j_n} \\
        & \leq 2^{-\left(h+\frac{5\varepsilon}{3}\right)pj_n}
    \end{align*}
    as expected, using the fact that
    \[
    -\left(\alpha_s-\frac{1}{p}-\varepsilon_s\right)\left(A+\frac{1}{m}\right) = -\left(\alpha_s-\frac{1}{p}\right)A-\frac{\alpha_s-\frac{1}{p}}{m}+\varepsilon_s\left(A+\frac{1}{m}\right) \leq -(h+2\varepsilon)+\frac{\varepsilon}{3}
    \]
    following from Conditions~\eqref{eq:epsilon_s1} and \eqref{eq:m_neg}.
\end{proof}

\begin{lemma}\label{lm:rid_order2}
    If $h<0$, then for every $s\in\{1,\ldots,S\}$ such that $A\geq\frac{h+2\varepsilon}{\alpha_s-\frac{1}{p}}$, for every $n\in\N$ large enough, if $\widetilde{\widetilde{\Lambda}}_{j_n}^{(p)}(h,\varepsilon)$ denotes the set of dyadic intervals $\lambda\in\widetilde{\Lambda}_{j_n}^{(p)}(h,\varepsilon)$ satisfying
    \[
    \#\left\{\lambda'\in\Lambda_{j_n'}:\lambda'\subseteq\lambda  \text{ and } 2^{-\left(\alpha_s-\frac{1}{p}+\varepsilon_s\right)j_n'} \leq \modu{c_{\lambda'}} \leq 2^{-\left(\alpha_s-\frac{1}{p}-\varepsilon_s\right)j_n'}\right\}<2^{-\left(h+\frac{11\varepsilon}{6}+\frac{1}{p}\right)pj_n}2^{\alpha_spj_n'},
    \]
    then
    \[
    \#\widetilde{\widetilde{\Lambda}}_{j_n}^{(p)}(h,\varepsilon) \geq 2^{\left(\rho_{\vec{c}}^{(p),\ast}(h)-4\delta\right)j_n}
    \]
    and for every $\lambda\in\widetilde{\widetilde{\Lambda}}_{j_n}^{(p)}(h,\varepsilon)$, one has
    \[
    \sum_{\lambda'\subseteq\lambda \,:\, 2^{-\left(\alpha_s-\frac{1}{p}+\varepsilon_s\right)j_n'} \leq \modu{c_{\lambda'}} \leq 2^{-\left(\alpha_s-\frac{1}{p}-\varepsilon_s\right)j_n'}} \modu{c_{\lambda'}}^p2^{-(j_n'-j_n)} \leq 2^{-\left(h+\frac{5\varepsilon}{3}\right)pj_n}.
    \]
\end{lemma}

\begin{proof}
    First, for all $\lambda\in \widetilde{\widetilde{\Lambda}}_{j_n}^{(p)}(h,\varepsilon)$, 
    \begin{align*}
        & \sum_{\lambda'\subseteq\lambda \,:\, 2^{-\left(\alpha_s-\frac{1}{p}+\varepsilon_s\right)j_n'} \leq \modu{c_{\lambda'}} \leq 2^{-\left(\alpha_s-\frac{1}{p}-\varepsilon_s\right)j_n'}} \modu{c_{\lambda'}}^p2^{-(j_n'-j_n)} \\
        & \leq \#\left\{\lambda'\subseteq\lambda :2^{-\left(\alpha_s-\frac{1}{p}+\varepsilon_s\right)j_n'} \leq \modu{c_{\lambda'}} \leq 2^{-\left(\alpha_s-\frac{1}{p}-\varepsilon_s\right)j_n'}\right\} 2^{-\left(\alpha_s-\frac{1}{p}-\varepsilon_s\right)pj_n'}2^{-(j_n'-j_n)} \\
        & \leq 2^{-\left(h+\frac{11\varepsilon}{6}+\frac{1}{p}\right)pj_n}2^{\alpha_spj_n'}2^{-(\alpha_s-\varepsilon_s)pj_n'}2^{j_n} \\
        & \leq 2^{-\left(h+\frac{11\varepsilon}{6}\right)pj_n}2^{\varepsilon_spj_n'} \\
        & \leq 2^{-\left(h+\frac{5\varepsilon}{3}\right)pj_n},
    \end{align*}
    using Condition~\eqref{eq:epsilon_s1}.

    Secondly, let us assume on the contrary that 
    \[
    \#\widetilde{\widetilde{\Lambda}}_{j_n}^{(p)}(h,\varepsilon) < 2^{\left(\rho_{\vec{c}}^{(p),\ast}(h)-4\delta\right)j_n}.
    \]
    Then, using Equation~\eqref{eq:epsilon_s2} and Lemma~\ref{lm:card_Lambdap},
    \begin{align*}
        2^{\left(\rho_{\vec{c}}\left(\alpha_s-\frac{1}{p}\right)+\delta\right)j_n'} 
        & \geq \#\left\{\lambda'\in\Lambda_{j_n'} :2^{-\left(\alpha_s-\frac{1}{p}+\varepsilon_s\right)j_n'} \leq \modu{c_{\lambda'}} \leq 2^{-\left(\alpha_s-\frac{1}{p}-\varepsilon_s\right)j_n'}\right\} \\
        & \geq \sum_{\lambda\in\widetilde{\Lambda}_{j_n}^{(p)}(h,\varepsilon)\setminus\widetilde{\widetilde{\Lambda}}_{j_n}^{(p)}(h,\varepsilon)}\#\left\{\lambda'\in\Lambda_{j_n'} :\lambda'\subseteq\lambda \text{ and } 2^{-\left(\alpha_s-\frac{1}{p}+\varepsilon_s\right)j_n'} \leq \modu{c_{\lambda'}} \leq 2^{-\left(\alpha_s-\frac{1}{p}-\varepsilon_s\right)j_n'}\right\} \\
        & \geq \#\left(\widetilde{\Lambda}_{j_n}^{(p)}(h,\varepsilon)\setminus\widetilde{\widetilde{\Lambda}}_{j_n}^{(p)}(h,\varepsilon)\right)2^{-\left(h+\frac{11\varepsilon}{6}+\frac{1}{p}\right)pj_n}2^{\alpha_spj_n'} \\
        & > \left(2^{\left(\rho_{\vec{c}}^{(p),\ast}(h)-3\delta\right)j_n}-2^{\left(\rho_{\vec{c}}^{(p),\ast}(h)-4\delta\right)j_n}\right)2^{-\left(h+2\varepsilon+\frac{1}{p}\right)pj_n}2^{\alpha_spj_n'} \\
        & = 2^{\left(\rho_{\vec{c}}^{(p),\ast}(h)-3\delta\right)j_n}\left(1-2^{-\delta j_n}\right)2^{-\left(h+2\varepsilon+\frac{1}{p}\right)pj_n}2^{\alpha_spj_n'} \\
        & \geq 2^{\left(\rho_{\vec{c}}^{(p),\ast}(h)-3\delta\right)j_n}2^{-\delta j_n}2^{-\left(h+2\varepsilon+\frac{1}{p}\right)pj_n}2^{\alpha_spj_n'},
    \end{align*}
    i.e.
    \[
    2^{\left(\rho_{\vec{c}}\left(\alpha_s-\frac{1}{p}\right)+\delta-\alpha_sp\right)j_n'} \geq 2^{\left(\rho_{\vec{c}}^{(p),\ast}(h)-4\delta-(hp+1)-2\varepsilon p\right)j_n}.
    \]
    However, this inequality cannot hold. Indeed, by Equation~\eqref{eq:bound_rho}
and the condition $\delta<\eta_f(p)$, we have
\[
\rho_{\vec{c}}\left(\alpha_s-\frac{1}{p}\right)
+\delta-\alpha_sp<0.
\]
Since $j_n'\geq Aj_n$, it follows that
\[
2^{\left(\rho_{\vec{c}}\left(\alpha_s-\frac{1}{p}\right)
+\delta-\alpha_sp\right)j_n'}
\leq
2^{\left(\rho_{\vec{c}}\left(\alpha_s-\frac{1}{p}\right)
+\delta-\alpha_sp\right)Aj_n}.
\]
Moreover,
\small
\begin{align*}
&\left(\rho_{\vec{c}}\left(\alpha_s-\frac{1}{p}\right)
+\delta-\alpha_sp\right)A \\
&\qquad \leq
\left(\rho_{\vec{c}}\left(\alpha_s-\frac{1}{p}\right)
+\delta-\alpha_sp\right)
\frac{h+2\varepsilon}{\alpha_s-\frac{1}{p}} \\
&\qquad \leq
\left(
\left(
\rho_{\vec{c}}^{(p),\ast}(h)-1
-\left(
4+
\frac{h\left(h+2\varepsilon+\frac{1}{p}\right)}
{\left(\alpha_0+\frac{1}{p}\right)(h+2\varepsilon)}
\right)\delta
-M\varepsilon p
\right)
\frac{\alpha_s-\frac{1}{p}}{h}
+1+\delta-\alpha_sp
\right)
\frac{h+2\varepsilon}{\alpha_s-\frac{1}{p}} \\
&\qquad \leq
\left(\rho_{\vec{c}}^{(p),\ast}(h)-1-4\delta-M\varepsilon p\right)
\left(1+\frac{2\varepsilon}{h}\right)
-(h+2\varepsilon)p \\
&\qquad =
\rho_{\vec{c}}^{(p),\ast}(h)-(hp+1)-4\delta-2\varepsilon p-M\varepsilon p
+\frac{2\varepsilon}{h}
\left(\rho_{\vec{c}}^{(p),\ast}(h)-1-4\delta-M\varepsilon p\right) \\
&\qquad \leq
\rho_{\vec{c}}^{(p),\ast}(h)-(hp+1)-4\delta-2\varepsilon p.
\end{align*}
\normalsize
Here, the first two estimates follow from Condition~\eqref{eq:delta} and
from the definition of $I$, which ensures that
\[
\frac{h+2\varepsilon}{\alpha_s-\frac{1}{p}}
\leq
\frac{h+2\varepsilon+\frac{1}{p}}
{\alpha_0+\frac{1}{p}}.
\]
The last inequality follows from Relation~\eqref{eq:epsilon_neg} together
with~\eqref{eq:delta}, since
\[
M\varepsilon p+\frac{Mhp}{2}+1
<0<
\rho_{\vec{c}}^{(p),\ast}(h)-4\delta.
\]
This yields the desired contradiction.
\end{proof}

\begin{rem}\label{rem:hyp}
    The previous lemma is the only result in this section whose proof relies on Hypothesis~\eqref{eq:hyp}, through the last condition on $\delta$ in~\eqref{eq:delta}. All the other results of the section therefore remain valid without this hypothesis.
\end{rem}

\begin{lemma}\label{lm:sum_order}
For every $n\in\N$ large enough and every $\lambda\in\widetilde{\widetilde{\Lambda}}_{j_n}^{(p)}(h,\varepsilon)$, we have
\[
2^{\left(r_\lambda(l_0(\lambda))A-\left(l_0(\lambda)\beta+\frac{1}{m}\right)p\left(A+\frac{1}{m}\right)\right)j_n} \leq \sum_{\lambda'\subseteq\lambda}\modu{c_{\lambda'}}^p2^{-(j'_n-j_n)}\leq 2^{\left(r_\lambda(l_0(\lambda))A-l_0(\lambda)\beta pA+\frac{1}{m}+\frac{\varepsilon}{2}p\right)j_n}.
\]\normalsize
\end{lemma}

\begin{proof}
The lower bound simply follows from the fact that there exist $2^{r_\lambda(l_0(\lambda))j_n'-j_n}$ coefficients $c_{\lambda'}$ that satisfy
\[
\modu{c_{\lambda'}}^p2^{-(j'_n-j_n)}\geq 2^{-\left(l_0(\lambda)\beta+\frac{1}{m}\right)pj'_n+j_n},
\]
with $Aj_n\leq j'_n\leq \left(A+\frac{1}{m}\right)j_n$.

To obtain the upper bound, we partition the set of dyadic intervals $\lambda'$ included in $\lambda$ according to the order of $\modu{c_{\lambda'}}$ (following Lemma~\ref{lm:list_orders}), which allows to write, thanks to Lemmas~\ref{lm:rid_order1} and \ref{lm:rid_order2},
\begin{align*}
\sum_{\lambda'\subseteq\lambda}\modu{c_{\lambda'}}^p2^{-(j'_n-j_n)}
& \leq (L+S+1)2^{\max\left(\sup_{l\in\{1,\ldots,L\}}\left(r_\lambda(l)A-l\beta pA+\frac{1}{m}\right),\,-(h+2\varepsilon)p,-\left(h+\frac{5\varepsilon}{3}\right)p\right) j_n}
\end{align*}
and we use both Equation~\eqref{sum_order} and the definition of $l_0(\lambda)$ to conclude the proof.
\end{proof}

Moreover, we can provide a lower and an upper bound for $r_\lambda(l_0(\lambda))$, which follow from Lemma~\ref{lm:sum_order} and Equation~\eqref{sum_order}.

\begin{cor}\label{cor:r_lambda}
For every $n\in\N$ large enough and every $\lambda\in\widetilde{\widetilde{\Lambda}}_{j_n}^{(p)}(h,\varepsilon)$, we have
\[
 \frac{l_0(\lambda)\beta pA-\frac{1}{m}-(h+2\varepsilon)p}{A} \leq r_\lambda(l_0(\lambda)) \leq \frac{\left(l_0(\lambda)\beta+\frac{1}{m}\right)\left(A+\frac{1}{m}\right)p-(h-\varepsilon)p}{A}.
\]
\end{cor}

In order to bound $\rho_{\vec{c}}^{(p),\ast}(h)$ by $\left(h+\frac{1}{p}\right)\frac{\rho_{\vec{c}}(\alpha)}{\alpha+\frac{1}{p}}$, where $\alpha$ is an exponent to be determined, we need to derive such an exponent from the collection
\[
\left\{l_0(\lambda) : \lambda\in\widetilde{\widetilde{\Lambda}}_{j_n}^{(p)}(h,\varepsilon), \, n\in\N\right\}
\]
of orders describing the dominating behaviours of different $p$-leaders in $\widetilde{\widetilde{\Lambda}}_{j_n}^{(p)}(h,\varepsilon)$. This is the aim of the following Lemma, for which we assume that $n\in\N$ is large enough to satisfy $L<2^{\delta j_n}$ if $\delta>0$.

\begin{lemma}\label{lm:l}
There exists $l\in\{1,\ldots,L\}$ such that for every $n\in\N$ large enough, at least $2^{(\rho_{\vec{c}}^{(p),\ast}(h)-5\delta)j_n}$ dyadic intervals $\lambda\in\widetilde{\widetilde{\Lambda}}_{j_n}^{(p)}(h,\varepsilon)$
satisfy $l_0(\lambda)=l$. Moreover, for every such interval $\lambda$,
one has \begin{equation}\label{r_lambda}
\max\left(\frac{l\beta pA-\frac{1}{m}-(h+2\varepsilon)p}{A},\frac{1}{A+\frac{1}{m}}\right) \leq r_{\lambda}(l) \leq \frac{\left(l\beta+\frac{1}{m}\right)\left(A+\frac{1}{m}\right)p-(h-\varepsilon)p}{A}.
\end{equation}
\end{lemma}

\begin{proof}
We know from Lemma~\ref{lm:rid_order2} that
\[
2^{(\rho_{\vec{c}}^{(p),\ast}(h)-4\delta)j_n} 
\leq \#\widetilde{\widetilde{\Lambda}}_{j_n}^{(p)}(h,\varepsilon) = \sum_{l=1}^L \#\widetilde{\widetilde{\Lambda}}_{j_n}^{(p)}(h,\varepsilon,l),
\]
where
$$
\widetilde{\widetilde{\Lambda}}_{j_n}^{(p)}(h,\varepsilon,l)=\left\{\lambda\in\widetilde{\widetilde{\Lambda}}_{j_n}^{(p)}(h,\varepsilon):l_0(\lambda)=l\right\}.
$$
Therefore, for every $n\in\N$,
$
2^{(\rho_{\vec{c}}^{(p),\ast}(h)-4\delta)j_n}\leq L\sup_{l\in\{1,\ldots,L\}}\#\widetilde{\widetilde{\Lambda}}_{j_n}^{(p)}(h,\varepsilon,l),$
from which follows the existence of $l_n\in\{1,\ldots,L\}$ such that
\[
2^{(\rho_{\vec{c}}^{(p),\ast}(h)-5\delta)j_n}\leq\#\widetilde{\widetilde{\Lambda}}_{j_n}^{(p)}(h,\varepsilon,l_n).
\]
Using the pigeon hole principle, 
we may assume that there exist $l\in\{1,\ldots,L\}$ such that for every $n\in\N$, 
\[
2^{(\rho_{\vec{c}}^{(p),\ast}(h)-5\delta)j_n}\leq \#\widetilde{\widetilde{\Lambda}}_{j_n}^{(p)}(h,\varepsilon,l),
\]
which is the condition requested in the first statement. Finally, Equation~\eqref{r_lambda} directly follows from Corollary~\ref{cor:r_lambda}.
\end{proof}

We may now conclude. It remains to assume that the parameters are chosen such that
\footnotesize\begin{equation}\label{eq:m}
    m\geq \max\left(\frac{1}{p\varepsilon},\frac{h+\frac{1}{p}+2\varepsilon}{6\varepsilon},\frac{\left(2+p\left(h+\frac{1}{p}-2\varepsilon\right)+p\left\lceil\frac{h+\frac{1}{p}+2\varepsilon}{\alpha_0+\frac{1}{p}}\right\rceil-\frac{\rho_{\vec{c}}^{(p),\ast}(h)-4p\varepsilon}{h+\frac{1}{p}+2\varepsilon}\right)\left(h+\frac{1}{p}+2\varepsilon\right)}{4\varepsilon(\rho_{\vec{c}}^{(p),\ast}(h)-4p\varepsilon)}\right).
\end{equation}\normalsize

Let us prove a technical lemma.

\begin{lemma}\label{lm:tech}
If $\rho_{\vec{c}}^{(p),\ast}(h)>0$, then 
\[
\rho_{\vec{c}}^{(p),\ast}(h)-5\delta+l\beta pA-\frac{1}{m}-(h+2\varepsilon)p-1\geq \frac{\left(A+\frac{1}{m}\right)(\rho_{\vec{c}}^{(p),\ast}(h)-4p\varepsilon)l\beta}{h+\frac{1}{p}+2\varepsilon}
\]
for any pair $(A,l)\in\mathcal{A}\times\{1,\ldots,L\}$ satisfying Equation~\eqref{r_lambda}.
\end{lemma}

\begin{proof}
From \eqref{r_lambda}, we know that the pair $(A,l)$ satisfies
\begin{equation}\label{cond_l}
\frac{1}{A+\frac{1}{m}} \leq \frac{\left(l\beta+\frac{1}{m}\right)\left(A+\frac{1}{m}\right)p-(h-\varepsilon)p}{A}.
\end{equation}
We must prove that the function
\[
l\mapsto \rho_{\vec{c}}^{(p),\ast}(h)-5\delta+l\beta pA-\frac{1}{m}-(h+2\varepsilon)p-1-\frac{\left(A+\frac{1}{m}\right)(\rho_{\vec{c}}^{(p),\ast}(h)-4p\varepsilon)l\beta}{h+\frac{1}{p}+2\varepsilon}
\]
is non-negative for every authorized $A\in\mathcal{A}$. Direct computations show that this function must be non-decreasing, otherwise Lemma~\ref{sup_B} would be contradicted. As a consequence, it reaches its minimum when $l$ is minimal, i.e. when
\[
l=\frac{A}{\beta\left(A+\frac{1}{m}\right)^2p}-\frac{1}{m\beta}+\frac{h-\varepsilon}{\beta\left(A+\frac{1}{m}\right)}
\]
in view of Conditions~\eqref{cond_l}.
Using the inequality 
\[
\frac{A}{\beta\left(A+\frac{1}{m}\right)^2p}-\frac{1}{m\beta}+\frac{h-\varepsilon}{\beta\left(A+\frac{1}{m}\right)}\geq\frac{h+\frac{1}{p}-2\varepsilon}{\beta\left(A+\frac{1}{m}\right)}-\frac{1}{m\beta},
\]
the minimum is eventually shown to be non-negative.
\end{proof}

We now have all the necessary tools to complete the proof.

\begin{proof}[Proof of Equation~\eqref{eq:thesis_sup}]
    From Lemma \ref{lm:l}, we know that for every $n\in\N$ large enough, at least 
    \[
    \max\Big(2^{\big(\rho_{\vec{c}}^{(p),\ast}(h)-5\delta\big)j_n}2^{\big(l\beta pA-\frac{1}{m}-(h+2\varepsilon)p\big)j_n-j_n},1\Big)
    \]
    coefficients at scale $j_n'$ satisfy
    $
    2^{-\big(l\beta-\frac{1}{p}+\frac{1}{m}\big)j'_n}\leq \modu{c_{\lambda'}}\leq 2^{-\big(l\beta-\frac{1}{p}-\frac{1}{m}\big) j'_n}
    $
    From Lemma~\ref{lm:tech} then follows that we have
    \[
    \frac{\rho_{\vec{c}}\left(l\beta-\frac{1}{p}\right)}{l\beta}\geq\left\{\begin{array}{ll}
      0   &  \text{ if } \rho_{\vec{c}}^{(p),\ast}(h)=0, \\
      \frac{\rho_{\vec{c}}^{(p),\ast}(h)-4p\varepsilon}{h+\frac{1}{p}+2\varepsilon}   & \text{ if } \rho_{\vec{c}}^{(p),\ast}(h)>0.
    \end{array}\right.
    \]
    Since $l\beta-\frac{1}{p}\in\ocint{\frac{-1}{p},h+2\varepsilon}$, we have in both cases
    $$
    \Big(h+2\varepsilon+\frac{1}{p}\Big)\sup_{\alpha\in\ocint{\frac{-1}{p},\,h+2\varepsilon}}\frac{\rho_{\vec{c}}(\alpha)}{\alpha+\frac{1}{p}}
     \geq \Big(h+2\varepsilon+\frac{1}{p}\Big)\frac{\rho_{\vec{c}}\Big(l\beta-\frac{1}{p}\Big)}{l\beta} 
     \geq \max(\rho_{\vec{c}}^{(p),\ast}(h)-4p\varepsilon,0).
 $$
    The conclusion then follows by letting $\varepsilon$ tend towards 0.
\end{proof}

\section{Study of the $p$-spectrum of Random Wavelet Series}\label{secRWStot}

The aim of this section is to establish Theorem~\ref{main_RWS}, namely to show that the $p$-large deviation wavelet formalism is satisfied by a broad class of random functions.

To this end, we study \textit{Random Wavelet Series}. These are random functions whose wavelet coefficients are sampled independently at each scale according to a prescribed distribution. Thus, their asymptotic distribution in size can be controlled through the underlying probability laws, while independence ensures that the coefficients are spatially distributed in a sufficiently homogeneous way over $[0,1]$. This makes Random Wavelet Series a natural model for investigating whether the upper bound obtained in the previous section is optimal.

Random Wavelet Series were introduced and studied by Aubry and Jaffard in~\cite{RWS}, who showed in particular that the statistical distribution of their coefficients determines deterministic wavelet density and profile functions, from which their almost sure multifractal spectrum can be computed. We begin by recalling the relevant definitions and known results before extending this approach to the $p$-spectrum.

\subsection{Random Wavelet Series}\label{secRWS}

A \textit{Random Wavelet Series} (RWS) is a process whose wavelet coefficients are drawn independently at each scale according to a prescribed distribution on a probability space $(\Omega,\mathscr{F},\prob)$. If
\[
f=\sum_{j\in\N_0}\sum_{k=0}^{2^j-1}c_{j,k}\psi_{j,k}
\]
is a RWS, then $X_{j,k}$ denotes the random variable $X_{j,k}=\frac{-\log_2|c_{j,k}|}{j}$,
and $\bm{\rho_j}$ denotes the common distribution of the $2^j$ random variables
$X_{j,k}$, $k\in\{0,\ldots,2^j-1\}$. Thus,
\[
\prob\left(|c_{j,k}|\geq 2^{-\alpha j}\right)
=
\bm{\rho_j}(\ocint{-\infty,\alpha}).
\]
For every $\alpha\in\R$, we set
\[
\bm{\rho}(\alpha)=
\lim_{\varepsilon\rightarrow0^+}
\limsup_{j\rightarrow+\infty}
\frac1j 
\log_2\left(
2^j\bm{\rho_j}(\ccint{\alpha-\varepsilon,\alpha+\varepsilon})
\right),
\]
and
\[
\bm{\nu}(\alpha)=
\lim_{\varepsilon\rightarrow0^+}
\limsup_{j\rightarrow+\infty}
\frac1j \log_2\left(
2^j\bm{\rho_j}(\ocint{-\infty,\alpha+\varepsilon})
\right).
\]
We also associate with the RWS the set
\[
W=
\left\{
\alpha\in\R:
\forall\varepsilon>0,\,
\sum_{j\in\N_0}
2^j\bm{\rho_j}
(\ccint{\alpha-\varepsilon,\alpha+\varepsilon})
=+\infty
\right\},
\]
and define $h_{\min}=\inf W$.
Throughout this paper, we restrict attention to Random Wavelet Series satisfying $\{\alpha\in\R:\bm{\rho}(\alpha)>0\}\neq\emptyset$, $\bm{\rho}(0)<1$, $\bm{\nu}(\alpha)>0$ for every $\alpha>h_{\min}$. Under these assumptions, $W\neq\emptyset$; since $W$ is closed,
$h_{\min}\in W$.

\medskip

We now recall how the wavelet density $\rho_{\vec{c}}$ and the wavelet profile $\nu_{\vec{c}}$ of a Random Wavelet Series $f$ are related to their theoretical counterparts $\bm{\rho}$ and $\bm{\nu}$, as established in~\cite{RWS}. The relation for the wavelet density is stated in Proposition~\ref{propW}, while the corresponding relation for the wavelet profile follows as Corollary~\ref{eg_nu}.

\begin{prop}\label{propW}
\cite{RWS} The following properties are satisfied:
\begin{enumerate}
\item $\bm{\rho}(\alpha)>0\Rightarrow\alpha\in W$ and $\bm{\rho}(\alpha)<0\Rightarrow\alpha\notin W$,
\item almost surely, for every $\alpha\in\R$,
\[
\rho_{\vec{c}}(\alpha)=\left\{\begin{array}{ll}
\bm{\rho}(\alpha) & \text{ if }\alpha\in W, \\
-\infty & \text{ otherwise}.
\end{array}\right.
\]
\end{enumerate}
\end{prop}

To derive Corollary~\ref{eg_nu}, we use the fact that $h_{\min}\in W$ and that $\bm{\nu}$ is non-decreasing, together with the fact that $\nu_{\vec c}$ and $\bm{\nu}$ are the increasing hulls of $\rho_{\vec c}$ and $\bm{\rho}$, respectively.

\begin{cor}\label{eg_nu}
\cite{RWS} The following properties are satisfied:
\begin{enumerate}
\item for every $\alpha\geq h_{\min}$, $\bm{\nu}(\alpha)\geq 0$,
\item almost surely, for every $\alpha\in\R$,
\[
\nu_{\vec{c}}(\alpha)=\left\{\begin{array}{ll}
\bm{\nu}(\alpha) & \text{ if }\alpha\geq h_{\min}, \\
-\infty & \text{ otherwise.}
\end{array}\right.
\]
\end{enumerate}
\end{cor}

In the classical setting considered in~\cite{RWS}, the multifractal spectrum of a Random Wavelet Series is studied under a uniform H\"older regularity assumption. Such a condition is ensured, for instance, by requiring the existence of $\gamma>0$ such that
\[
\bm{\rho}(\alpha)<0
\qquad\text{for every }\alpha<\gamma.
\]
In the present work, we allow negative regularities and therefore work in a broader setting. For a fixed $p>0$, the natural requirement is instead that the series has positive scaling function, namely $\eta_f(p)>0$ almost surely. This is ensured by the condition
\[
\bm{\rho}(\alpha)<\alpha p+1
\qquad\text{for every }\alpha<0,
\]
the corresponding inequality being automatically satisfied for $\alpha\geq0$. For a fixed Random Wavelet Series, this determines the range of admissible values of $p$. Namely, we consider
\[
0<p<p_0,
\qquad
p_0\coloneqq
\inf_{\alpha<0}
\frac{\bm{\rho}(\alpha)-1}{\alpha},
\]
which is precisely the range appearing in Theorem~\ref{main_RWS}.

\subsection{Proof of Theorem~\ref{main_RWS}}\label{secpRWS}

We consider a Random Wavelet Series
\[
f=\sum_{j\in\N_0}\sum_{k=0}^{2^j-1}c_{j,k}\psi_{j,k}
\]
such that 
\[
p_0\coloneqq\inf_{\alpha<0}\frac{\bm{\rho}(\alpha)-1}{\alpha}>0
\]
and we consider $p<p_0$. For every $\alpha\geq \frac{-1}{p}$ and every $\delta\in\ccint{0,1}$, let
\[
E(\alpha,\delta)=\limsup_{j\rightarrow+\infty}\bigcup_{k\in F_j(\alpha)}B\left(k2^{-j},2^{-\delta j+2\log_2j}\right),
\]
where 
\[
F_j(\alpha)=\{k\in\{0,\ldots,2^j-1\}:\modu{c_{j,k}}\geq 2^{-\alpha j}\}.
\]
Write
\[
h_{\max}^{(p)}=\inf\left\{h>\frac{-1}{p}:h+\frac{1}{p}=\left(\sup_{\alpha\in\ocint{\frac{-1}{p},h}}\frac{\bm{\rho}(\alpha)}{\alpha+\frac{1}{p}}\right)^{-1}\right\}.
\]
Moreover, let us define $\bm{\lambda}$ by
\[
\bm{\lambda}(\alpha)=\limsup_{j\rightarrow+\infty}\frac{1}{j}\log_2\left(2^j\bm{\rho_j}(\ocint{-\infty,\alpha})\right)
\]
for every $\alpha\in\R$, so that 
\[
\bm{\nu}(\alpha)=\lim_{\varepsilon\rightarrow0^+}\bm{\lambda}(\alpha+\varepsilon).
\]
Since $\bm{\lambda}$ is non-decreasing, the set $\mathcal{D}$ of its discontinuities is at most countable and $\bm{\nu}(\alpha)=\bm{\lambda}(\alpha)$ for every $\alpha\in\R\setminus\mathcal{D}$. 

\medskip

Theorem~\ref{main_RWS} states that, for Random Wavelet Series, the upper bound obtained for non-negative regularities extends to the whole range of admissible $p$-exponents and is sharp. We first establish the corresponding upper bound in Proposition~\ref{upperbound_RWS} (enhancing Theorem~\ref{petitsoleil}), for which the following lemma provides the key probabilistic estimate.

\begin{lemma}\label{lm:rid_RWS}
    Almost surely, for every $\alpha\in\R$ such that $\bm{\lambda}(\alpha)\geq 0$ and every $\varepsilon>0$, there exists $J\in\N$ such that for all $j\geq J$, all $\lambda\in\Lambda_j$, and all $j'\geq j$,
    \[
    \#\{\lambda'\in\Lambda_{j'}:\lambda'\subseteq\lambda\text{ and }\modu{c_{\lambda'}}\geq 2^{-\alpha j'}\}\leq \left\{\begin{array}{ll}
        2^{\varepsilon j'} & \text{ if }\bm{\lambda}(\alpha)=0 \\
        2^{\frac{\varepsilon}{\bm{\lambda}(\alpha)}j} & \text{ if } \bm{\lambda}(\alpha)>0 \text{ and } j'\leq\frac{j}{\bm{\lambda}(\alpha)}, \\
        2^{(\bm{\lambda}(\alpha)+\varepsilon)j'-j} & \text{ if } \bm{\lambda}(\alpha)>0 \text{ and } j'\geq\frac{j}{\bm{\lambda}(\alpha)}.
    \end{array}\right.
    \]
\end{lemma}

\begin{proof}
It is enough to prove the result for fixed $\alpha\in\R$ such that
$\bm{\lambda}(\alpha)\geq0$ and fixed $\varepsilon>0$. The
statement then follows by a countability argument, using the
non-decreasingness and right-continuity of $\bm{\lambda}$.
    For all $j\in\N$, all $\lambda\in\Lambda_j$ and all $j'\geq j$, let us write $X(\lambda,j')$ the random variable defined by
    \[
    X(\lambda,j')=\#\{\lambda'\in\Lambda_{j'}:\lambda'\subseteq\lambda\text{ and }\modu{c_{\lambda'}}\geq 2^{-\alpha j'}\}.
    \]
    Notice that $X(\lambda,j')\sim\text{Bin}\Big(2^{j'-j},\bm{\rho_{j'}}(\ocint{-\infty,\alpha})\Big)$. Moreover, there exists $J\in\N$ such that for all $j\geq J$ and all $j'\geq j$,
    \[
    \mathbb{E}[X(\lambda,j')] = 2^{j'-j}\bm{\rho_{j'}}(\ocint{-\infty,\alpha}) \leq 2^{\left(\bm{\lambda}(\alpha)+\frac{\varepsilon}{2}\right)j'-j},
    \]
    If $\bm{\lambda}(\alpha)=0$, Markov's inequality ensures that
    \[
    \prob\left(\exists\lambda\in\Lambda_j\;\exists j'\geq j\text{ such that } X(\lambda,j')>2^{\varepsilon j'}\right) \leq 2^j\sum_{j'\geq j}2^{-\frac{\varepsilon}{2}j'-j} = \frac{2^{-\frac{\varepsilon}{2}j}}{1-2^{-\frac{\varepsilon}{2}}}, 
    \]
    and Borel-Cantelli Lemma then settles this case. Let us now assume $\bm{\lambda}(\alpha)>0$.
    Using again Borel--Cantelli Lemma, it is enough to prove that the series
    \[
    \sum_{j\in\N}\prob\left(\exists\lambda\in\Lambda_j\;\exists j'\leq \frac{j}{\bm{\lambda}(\alpha)}\text{ such that } X(\lambda,j')>2^{\frac{\varepsilon}{\bm{\lambda}(\alpha)}j}\right)
    \]
    and
    \[
    \sum_{j\in\N}\prob\left(\exists\lambda\in\Lambda_j\;\exists j'\geq \frac{j}{\bm{\lambda}(\alpha)}\text{ such that } X(\lambda,j')>2^{(\bm{\lambda}(\alpha)+\varepsilon)j'-j}\right)
    \]
    converge. 
   Stute's concentration lemma~\cite{stute} states that
    \small\[
    \prob\left(\modu{X(\lambda,j')-\mathbb{E}[X(\lambda,j')]}\geq z\right)\leq 2\exp\left(-\mathbb{E}[X(\lambda,j')]\left[\left(1+\frac{z}{\mathbb{E}[X(\lambda,j')]}\right)\ln\left(1+\frac{z}{\mathbb{E}[X(\lambda,j')]}\right)-\frac{z}{\mathbb{E}[X(\lambda,j')]}\right]\right)
    \]\normalsize
    for all $z>0$. 
    Therefore, for all $j \leq j'\leq\frac{j}{\bm{\lambda}(\alpha)}$,
    \begin{align*}
        \prob\left(X(\lambda,j')>2^{\frac{\varepsilon}{\bm{\lambda}(\alpha)}j}\right)
        & = \prob\left(X(\lambda,j')-\mathbb{E}[X(\lambda,j')]>2^{\frac{\varepsilon}{\bm{\lambda}(\alpha)}j}-\mathbb{E}[X(\lambda,j')]\right) \\
        & \leq 2\exp\left(-\mathbb{E}[X(\lambda,j')]\left[\frac{2^{\frac{\varepsilon}{\bm{\lambda}(\alpha)}j}}{\mathbb{E}[X(\lambda,j')]}\ln\left(\frac{2^{\frac{\varepsilon}{\bm{\lambda}(\alpha)}j}}{\mathbb{E}[X(\lambda,j')]}\right)-\frac{2^{\frac{\varepsilon}{\bm{\lambda}(\alpha)}j}}{\mathbb{E}[X(\lambda,j')]}+1\right]\right) \\
        & \leq 2\exp\left(-\left[2^{\frac{\varepsilon}{\bm{\lambda}(\alpha)}j}\ln\left(\frac{2^{\frac{\varepsilon}{\bm{\lambda}(\alpha)}j}}{\mathbb{E}[X(\lambda,j')]}\right)-2^{\frac{\varepsilon}{\bm{\lambda}(\alpha)}j}
        \right]\right) \\
        & \leq 2\exp\left(-2^{\frac{\varepsilon}{\bm{\lambda}(\alpha)}j}\left[\ln\left(2^{\frac{\varepsilon}{2\bm{\lambda}(\alpha)}j}\right)-1
        \right]\right) \\
        & \leq 2\exp\left(-2^{\frac{\varepsilon}{\bm{\lambda}(\alpha)}j}\right),
    \end{align*}
    and for all $j'\geq\frac{j}{\bm{\lambda}(\alpha)}$,
    \begin{align*}
        \prob\left(X(\lambda,j')>2^{(\bm{\lambda}(\alpha)+\varepsilon)j'-j}\right)
        & = \prob\left(X(\lambda,j')-\mathbb{E}[X(\lambda,j')]>2^{(\bm{\lambda}(\alpha)+\varepsilon)j'-j}-\mathbb{E}[X(\lambda,j')]\right) \\
        & \leq 2\exp\left(-\mathbb{E}[X(\lambda,j')]\left[\frac{2^{(\bm{\lambda}(\alpha)+\varepsilon)j'-j}}{\mathbb{E}[X(\lambda,j')]}\ln\left(\frac{2^{(\bm{\lambda}(\alpha)+\varepsilon)j'-j}}{\mathbb{E}[X(\lambda,j')]}\right)-\frac{2^{(\bm{\lambda}(\alpha)+\varepsilon)j'-j}}{\mathbb{E}[X(\lambda,j')]}+1\right]\right) \\
        & \leq 2\exp\left(-2^{(\bm{\lambda}(\alpha)+\varepsilon)j'-j}\left[\ln\left(2^{\frac{\varepsilon}{2}j'}\right)-1
        \right]\right) \\
        & \leq 2\exp\left(-2^{(\bm{\lambda}(\alpha)+\varepsilon)j'-j}\right).
    \end{align*}
    As a consequence,
    \begin{align*}
        \prob\left(\exists\lambda\in\Lambda_j\;\exists j \le j'\leq \frac{j}{\bm{\lambda}(\alpha)}\text{ such that } X(\lambda,j')>2^{\frac{\varepsilon}{\bm{\lambda}(\alpha)}j}\right) \leq  2^{j+1}\,\frac{j}{\bm{\lambda}(\alpha)}\exp\left(-2^{\frac{\varepsilon}{\bm{\lambda}(\alpha)}j}\right),
    \end{align*}
    and
    \begin{align*}
        \prob\left(\exists\lambda\in\Lambda_j\;\exists j'\geq \frac{j}{\bm{\lambda}(\alpha)}\text{ such that } X(\lambda,j')>2^{\bm{\lambda}(\alpha)j'-j}\right) 
        & \leq 2^{j+1}\exp\left(-2^{(\bm{\lambda}(\alpha)+\varepsilon)\left\lfloor\frac{j}{\bm{\lambda}(\alpha)}\right\rfloor-j}\right) \\
        & \leq 2^{j+1}\exp\left(-2^{-(\bm{\lambda}(\alpha)+\varepsilon)}2^{\frac{\varepsilon}{\bm{\lambda}(\alpha)}j}\right).
    \end{align*}
Both right-hand sides are summable in $j$, and the conclusion follows
from the Borel--Cantelli lemma.
\end{proof}

\begin{prop}\label{upperbound_RWS}
    Almost surely, for all $0<p<p_0$ and all $h\geq\frac{-1}{p}$,
    \[
    \rho^{(p), \ast}_{\vec c}(h) \leq \left(h+\frac{1}{p}\right)\sup_{\alpha\in\ocint{\frac{-1}{p},h}}\frac{\rho_{\vec{c}}(\alpha)}{\alpha+\frac{1}{p}}.
    \]
\end{prop}

\begin{proof}
In view of Remark~\ref{rem:hyp}, it is enough to show that
Lemma~\ref{lm:rid_order2} remains valid without assuming~\eqref{eq:hyp}.
More precisely, we aim to prove that, almost surely, for every $p<p_0$,
every $s\in\{1,\ldots,S\}$, and all sufficiently large $n\in\N$,
$\widetilde{\Lambda}_{j_n}^{(p)}(h,\varepsilon)=\widetilde{\widetilde{\Lambda}}_{j_n}^{(p)}(h,\varepsilon)$, that is, for every $\lambda\in\widetilde{\Lambda}_{j_n}^{(p)}(h,\varepsilon)$,
    \[
    \#\left\{\lambda'\in\Lambda_{j_n'}:\lambda'\subseteq\lambda  \text{ and } 2^{-\left(\alpha_s-\frac{1}{p}+\varepsilon_s\right)j_n'} \leq \modu{c_{\lambda'}} \leq 2^{-\left(\alpha_s-\frac{1}{p}-\varepsilon_s\right)j_n'}\right\}<2^{-\left(h+\frac{11\varepsilon}{6}+\frac{1}{p}\right)pj_n}2^{\alpha_spj_n'}.
    \]
    To that end, for every $\alpha\in I$, we choose $\varepsilon(\alpha)$ as in~\eqref{eq:epsilon_s1}, small enough so that, in addition,
    \[
    \bm{\nu}\left(\alpha-\frac{1}{p}\right)+\frac{\varepsilon p}{12\left(A+\frac{1}{m}\right)}>\bm{\lambda}\left(\alpha-\frac{1}{p}+\varepsilon(\alpha)\right).
    \]
    Moreover, we consider the full probability event on which $\eta_f(p)>0$ for all $p<p_0$, $\bm{\nu}=\nu_{\vec{c}}$ on $\coint{h_{\min},+\infty}$ (in view of Corollary~\ref{eg_nu}), and Lemma~\ref{lm:rid_RWS} is satisfied. 
    Finally, let us notice that 
    \[
    2^{-\left(h+\frac{11\varepsilon}{6}+\frac{1}{p}\right)pj_n}2^{\alpha_spj_n'} \geq 2^{\frac{\varepsilon}{6}pj_n} \geq 2^{\frac{\varepsilon p}{6\left(A+\frac{1}{m}\right)}j'_n}
    \]
    and
    \begin{align*}
        2^{-\left(h+\frac{11\varepsilon}{6}+\frac{1}{p}\right)pj_n}2^{\alpha_spj_n'} 
        \geq & 2^{\left(\nu_{\vec{c}}\left(\alpha_s-\frac{1}{p}\right)+\frac{\varepsilon p}{6\left(A+\frac{1}{m}\right)}\right)j'_n-j_n}2^{\left(\left(\alpha_s-\frac{1}{p}\right)p+1-\nu_{\vec{c}}\left(\alpha_s-\frac{1}{p}\right)\right)j_n'}2^{-(h+2\varepsilon)pj_n} \\
        \geq & 2^{\left(\bm{\nu}\left(\alpha_s-\frac{1}{p}\right)+\frac{\varepsilon p}{6\left(A+\frac{1}{m}\right)}\right)j'_n-j_n}2^{\left(\left(\alpha_s-\frac{1}{p}\right)p+1-\nu_{\vec{c}}\left(\alpha_s-\frac{1}{p}\right)-(h+2\varepsilon)p\right)j_n} \\
        \geq & 2^{\left(\bm{\lambda}\left(\alpha_s-\frac{1}{p}+\varepsilon_s\right)+\frac{\varepsilon p}{12\left(A+\frac{1}{m}\right)}\right)j_n'-j_n}
    \end{align*}
    since
    \[
    \left(\alpha_s-\frac{1}{p}\right)p+1-\nu_{\vec{c}}\left(\alpha_s-\frac{1}{p}\right) \geq 0
    \]
    in view of the condition $\eta_f(p)>0$, and
    \[
    \Big(\big(\alpha_s-\frac{1}{p}\big)-(h+2\varepsilon)\Big)p+\Big(1-\nu_{\vec{c}}\big(\alpha_s-\frac{1}{p}\big)\Big) \geq 0.
    \]
    It follows from these inequalities that, if the result was false, then there would exist $p<p_0$, $s\in\{1,\ldots,S\}$, $n\in\N$ and $\lambda\in\widetilde{\Lambda}_{j_n}^{(p)}(h,\varepsilon)$ such that 
    \begin{align*}
        & \#\Bigg\{\lambda'\in\Lambda_{j_n'}:\lambda'\subseteq\lambda \text{ and }2^{-\big(\alpha_s-\frac{1}{p}+\varepsilon_s\big)j_n'} \leq \modu{c_{\lambda'}} \leq 2^{-\big(\alpha_s-\frac{1}{p}-\varepsilon_s\big)j_n'}\Bigg\} \\
        & \qquad \geq \max\Bigg(2^{\frac{\varepsilon}{6}pj_n}, 2^{\frac{\varepsilon p}{6\big(A+\frac{1}{m}\big)}j'_n}, 2^{\Big(\bm{\lambda}\big(\alpha_s-\frac{1}{p}+\varepsilon_s\big)+\frac{\varepsilon p}{12\big(A+\frac{1}{m}\big)}\Big)j_n'-j_n}\Bigg),
    \end{align*}
    which would contradict Lemma~\ref{lm:rid_RWS}.
\end{proof}

\medskip

Let us now turn to the lower bound. In order to determine the almost sure $p$-spectrum of $f$, we need to describe the sets of points sharing a given $p$-exponent and estimate their Hausdorff dimension. As shown in Lemma~\ref{E_exp} below, the sets $E(\alpha,\delta)$ introduced above play a central role in this description. By classical mass transference principles, the problem reduces to identifying the values of $\delta$ for which $E(\alpha,\delta)$ has full Lebesgue measure in $\ccint{0,1}$. This is established in Proposition~\ref{leb1}, inspired by a result in~\cite{RWS}. Its proof relies on Lemma~\ref{j_alpha}, which controls the range of scales at which one can guarantee, below any given dyadic interval, the existence of a wavelet coefficient of prescribed order. The following lemma and its proof are adapted from a corresponding result for Lacunary Wavelet Series; see~\cite{lacunaryCantor}.

\begin{lemma}\label{j_alpha}
Let $\alpha\geq\frac{-1}{p}$ be such that $\bm{\lambda}(\alpha)>0$. Almost surely, for every $\varepsilon>0$ satisfying $\bm{\lambda}(\alpha)>\varepsilon$, for infinitely many scales $j$ and for all $\lambda\in\Lambda_j$, the smallest scale $j_\alpha(\lambda)\geq j$ for which there exists $\lambda'\in\Lambda_{j_\alpha(\lambda)}$ such that $\lambda'\subseteq\lambda$ and $\modu{c_{\lambda'}}\geq 2^{-\alpha j_\alpha(\lambda)}$ satisfies
\[
j_{\alpha}(\lambda)\leq \left\lceil\frac{1}{\bm{\lambda}(\alpha)-\varepsilon}(j+\log_2j)\right\rceil.
\]
\end{lemma}

\begin{proof}
It is enough to prove the result for a fixed $\varepsilon>0$ such that $\bm{\lambda}(\alpha)-\varepsilon>0$, the conclusion for every such $\varepsilon$ following by a countability argument. By the definition of $\bm{\lambda}(\alpha)$, we can choose an increasing sequence $(J_n)_{n\in\N}$ such that for all $n\in\N$,
\[
2^{J_n}\bm{\rho}_{J_n}(\ocint{-\infty,\alpha})
>
2^{(\bm{\lambda}(\alpha)-\varepsilon)J_n}.
\]
For every $n \in \N$, set
\[
j_n=\max\left\{j\in\N:\left\lceil\frac{1}{\bm{\lambda}(\alpha)-\varepsilon}(j+\log_2j)\right\rceil \leq J_n\right\}
\]
 and let $A_n$ be the event that there exists a dyadic interval $\lambda\in\Lambda_{j_n}$ containing no descendant $\lambda'$ of scale at most $J_n$ satisfying
$\modu{c_{\lambda'}}\geq2^{-\alpha j'}$. Then
\begin{align*}
\prob(A_n)
& \leq \sum_{\lambda\in\Lambda_{j_n}}\prob\left(\forall \lambda'\in\Lambda_{J_n}\text{ with }\lambda'\subseteq\lambda,\text{ one has }\modu{c_{\lambda'}}<2^{-\alpha J_n}\right) \\
& \leq 2^{j_n}\left(1-\bm{\rho_{J_n}}(\ocint{-\infty,\alpha})\right)^{2^{J_n-j_n}} \\
& \leq 2^{j_n}\exp\left(-2^{J_n-j_n}\bm{\rho_{J_n}}(\ocint{-\infty,\alpha})\right) \\
& \leq \left(\frac{2}{e}\right)^{j_n}
\end{align*}
by the choices of $J_n$ and $j_n$.
This establishes the convergence of the series $\sum_{n \in \mathbb{N}} \mathbb{P}(A_n)$, and the Borel-Cantelli Lemma concludes the proof.
\end{proof}

\begin{prop}\label{leb1}
For every $\alpha\geq\frac{-1}{p}$ such that $\alpha\notin\mathcal{D}$ and $\bm{\nu}(\alpha)>0$, almost surely, for all $\varepsilon>0$ satisfying $\bm{\nu}(\alpha)>\varepsilon$,
\[
\ccint{0,1}\subseteq E(\alpha,\bm{\nu}(\alpha)-\varepsilon).
\]
\end{prop}

\begin{proof}
Fix such an $\alpha$,
and consider the full probability event given by Lemma~\ref{j_alpha}. Clearly, \mbox{$\bm{\lambda}(\alpha)=\bm{\nu}(\alpha)>0$}. Then for every fixed $\varepsilon>0$ satisfying $\bm{\nu}(\alpha)>\varepsilon$, there exists a sequence $(j_n)_{n\in\N}$ such that for all $n\in\N$ and all $\lambda\in\Lambda_{j_n}$,
\[
j_\alpha(\lambda)\leq \left\lceil\frac{1}{\bm{\lambda}(\alpha)-\varepsilon}(j_n+\log_2j_n)\right\rceil.
\]
In particular, for all $n\in\N$ and all $x\in\ccint{0,1}$, there exist a scale $J_n$ satisfying
\[
j_n\leq J_n\leq \left\lceil\frac{1}{\bm{\lambda}(\alpha)-\varepsilon}(j_n+\log_2j_n)\right\rceil
\]
and a position $K_n\in F_{J_n}(\alpha)$ such that $\lambda_{J_n,K_n}\subseteq \lambda_{j_n}(x)$, in which case
\[
|x-K_n 2^{-J_n}|< 2^{-j_n} \leq 2^{-(\boldsymbol{\lambda}(\alpha) - \varepsilon)(J_n-1) + \log_2 j_n} \leq  2^{-(\boldsymbol{\nu}(\alpha) - \varepsilon)J_n + 2\log_2 J_n}.
\]
This shows that any $x\in\ccint{0,1}$ belongs to $E(\alpha,\bm{\nu}(\alpha)-\varepsilon)$, as expected.
\end{proof}

As announced, the following Lemma identifies an upper bound for the $p$-exponents of points belonging to $E(\alpha,\delta)$. 

\begin{lemma}\label{E_exp}
For every $\alpha\geq\frac{-1}{p}$ and every $\delta\in\ocint{0,1}$,
\[
E(\alpha,\delta)\subseteq \Bigg\{x\in\ccint{0,1}:h^{(p)}_f(x)\leq \frac{\alpha+\frac{1}{p}}{\delta}-\frac{1}{p}\Bigg\}.
\]
\end{lemma}

\begin{proof}
Fix $\alpha\geq\frac{-1}{p}$, $\delta\in\ocint{0,1}$ and $x\in E(\alpha,\delta)$. By definition, there exists a sequence $(j_n)_{n\in\N}$ such that for every $n\in\N$, there exists $k_n\in F_{j_n}(\alpha)$ satisfying 
\[
\modu{x-k_n2^{-j_n}}<2^{-\delta j_n+2\log_2j_n}.
\]
For each $n\in\N$, we fix $j'_n=\lfloor\delta j_n-2\log_2j_n\rfloor$, so that
$\lambda_{j_n,k_n}\subseteq 3\lambda_{j'_n}(x)$. 
It follows that for every $\varepsilon>0$, if $n$ is large enough, then
\[
l_{\lambda_{j'_n}(x)}^{(p)}\geq 2^{-\alpha j_n}2^{-\frac{j_n-j'_n}{p}}
\geq 2^{-\left(\alpha+\frac{1}{p}\right)\frac{j'_n}{\delta-\varepsilon}}2^{\frac{j'_n}{p}}.
\]
Since $j'_n\rightarrow+\infty$ when $n\rightarrow+\infty$, we obtain
$
h_f^{(p)}(x)\leq \frac{\alpha+\frac{1}{p}}{\delta-\varepsilon}-\frac{1}{p},
$
and the conclusion follows.
\end{proof}

As a straighforward consequence, we get the following inclusion.

\begin{cor}\label{E_exp_2}
For every $h>\frac{-1}{p}$,
\[
\bigcup_{\alpha\in\ocint{\frac{-1}{p},h}}E\left(\alpha,\frac{\alpha+\frac{1}{p}}{h+\frac{1}{p}}\right)\subseteq \big\{x\in\ccint{0,1}:h^{(p)}_f(x)\leq h\big\}.
\]
\end{cor}

The proof of Theorem~\ref{main_RWS} is now based on three main results: Proposition~\ref{reg_max} deals with the case $h>h_{\max}^{(p)}$ and Proposition~\ref{palier} handles the value $h_{\min}$, while Theorem~\ref{inf_spectre} relies on the general mass transference principle stated in Theorem~\ref{edouard} to obtain the essential part of the spectrum, to identify when $h$ belongs to the interval $[h_{\min},h_{\max}^{(p)}]$. Notice that the case $h<h_{\min}$ is a straightforward consequence of Theorem~\ref{mainSup} and Corollary~\ref{eg_nu}.

\medskip

Let us start by showing that $h_{\max}^{(p)}$ is the maximal regularity, a result mentioned in \cite{RWS} in the case of the H\"older regularity.

\begin{prop}\label{reg_max}
Almost surely, for all $p<p_0$ and all $h>h_{\max}^{(p)}$,
$
\mathscr{D}^{(p)}_f(h)=-\infty.
$
\end{prop}

\begin{proof}
For a fixed $p<p_0$, let us show that, almost surely, for every $x\in\ccint{0,1}$, $h_f^{(p)}(x)\leq h_{\max}^{(p)}$. By definition, 
\[
h_{\max}^{(p)}+\frac{1}{p}
\geq\inf_{\alpha>\frac{-1}{p}\,:\,\bm{\rho}(\alpha)>0}\frac{\alpha+\frac{1}{p}}{\bm{\rho}(\alpha)}.
\]
Then for every $\varepsilon>0$, there exists $\alpha_\varepsilon>\frac{-1}{p}$ such that
\[
\bm{\rho}(\alpha_\varepsilon)>0 \quad\text{and}\quad h_{\max}^{(p)}+\frac{1}{p}+\varepsilon>\frac{\alpha_\varepsilon+\frac{1}{p}}{\bm{\rho}(\alpha_\varepsilon)}.
\]
Fix $\varepsilon>0$ and $\delta>0$ such that $\bm{\rho}(\alpha_\varepsilon)>\delta$. Since $0<\bm{\rho}(\alpha_\varepsilon)\leq\bm{\lambda}(\alpha_\varepsilon+\delta)$, by Lemma~\ref{j_alpha}, almost surely, at infinitely many scales $j$,
\[
j_{\alpha_\varepsilon+\delta}(\lambda)\leq \left\lceil\frac{1}{\bm{\rho}(\alpha_\varepsilon)-\delta}(j+\log_2j)\right\rceil\quad\forall\lambda\in\Lambda_j.
\]
As a consequence, almost surely, to every $x\in\ccint{0,1}$ and to infinitely many scales $j$, it is possible to associate $J_j(x)\in\N$ such that 
\[
j\leq J_j(x)\leq \left\lceil\frac{1}{\bm{\rho}(\alpha_\varepsilon)-\delta}(j+\log_2j)\right\rceil
\]
and
\[
l_{\lambda_j(x)}^{(p)}\geq 2^{-(\alpha_\varepsilon+\delta)J_j(x)}2^{-\frac{J_j(x)-j}{p}}\geq 2^{-\left(\alpha_\varepsilon+\delta+\frac{1}{p}\right)\left(\frac{1}{\bm{\rho}(\alpha_\varepsilon)-\delta}(j+2\log_2j)\right)}2^{\frac{j}{p}}.
\]
Therefore, almost surely, for every $x\in\ccint{0,1}$,
\begin{align*}
h_f^{(p)}(x)
\leq \frac{\alpha_\varepsilon+\frac{1}{p}+\delta}{\bm{\rho}(\alpha_\varepsilon)-\delta}-\frac{1}{p} 
\leq h_{\max}^{(p)}+\varepsilon+\frac{\bm{\rho}(\alpha_\varepsilon)+\alpha_\varepsilon+\frac{1}{p}}{\bm{\rho}(\alpha_\varepsilon)(\bm{\rho}(\alpha_\varepsilon)-\delta)}\,\delta.
\end{align*}
Considering sequences $(\delta_n)_{n\in\N}$ and $(\varepsilon_n)_{n\in\N}$ that converge to 0,
we get that, almost surely, for every $x\in\ccint{0,1}$,
\[
h_f^{(p)}(x)\leq 
h_{\max}^{(p)}.
\]
To ensure that the full-probability event does not depend on $p$, let  $(p_n)_{n\in\N}$ be a dense sequence in $\ooint{0,p_0}$. Then, almost surely, for every $p<p_0$ and every $x\in\ccint{0,1}$, if $(p_n')_{n\in\N}$ is an increasing subsequence of $(p_n)_{n\in\N}$ converging to $p$, we have 
\[
h_f^{(p)}(x)\leq h_f^{(p_n')}(x)\leq h_{\max}^{(p_n')}
\]
for every $n\in\N$, which suffices.
\end{proof}

Let us now prove that the minimal regularity $h_{\min}$ is reached. This result is only useful when $\bm{\nu}(h_{\min})= 0$, otherwise it follows easily from Remark~\ref{rem:h_min}. 

\begin{prop}\label{palier}
Almost surely, for all $p<p_0$,
$
\mathscr{D}_f^{(p)}(h_{\min})\geq 0.
$
\end{prop}

\begin{proof}
Let us show that almost surely, for all $p<p_0$, there exists $x\in\ccint{0,1}$ for which \mbox{$h_f^{(p)}(x)=h_{\min}$}. 
Since $h_{\min}\in W$, for every $\varepsilon>0$ and every $j\in\mathbb N$,
there exists $J_j(\varepsilon)>j$ such that
\[
\sum_{j'=j+1}^{J_j(\varepsilon)}
2^{j'}\rho_{j'}([h_{\min}-\varepsilon,h_{\min}+\varepsilon])
\geq 2^{2j}.
\]
Let $\Omega(j,\varepsilon)$ be the event that every
$\lambda\in\Lambda_j$ contains a descendant $\lambda'\in\Lambda_{j'}$
for some $j<j'\leq J_j(\varepsilon)$ such that
$
|c_{\lambda'}|\geq 2^{-(h_{\min}+\varepsilon)j'}.
$
By independence,
\[
\begin{aligned}
\mathbb P\bigl(\Omega(j,\varepsilon)^c\bigr)
&\leq
2^j
\prod_{j'=j+1}^{J_j(\varepsilon)}
\left(
1-\rho_{j'}((-\infty,h_{\min}+\varepsilon])
\right)^{2^{j'-j}}\\
&\leq
2^j\exp\Big(
-2^{-j}
\sum_{j'=j+1}^{J_j(\varepsilon)}
2^{j'}\rho_{j'}([h_{\min}-\varepsilon,h_{\min}+\varepsilon])
\Big)\\
&\leq
2^j e^{-2^j}.
\end{aligned}
\]
By the Borel--Cantelli lemma and a countability argument, there exists a full-probability event on which, for every $n\geq1$, $\Omega(j,\frac1n)$ holds for all sufficiently large $j$.
On this full-probability event, we construct recursively a sequence $(\lambda_n)_{n\geq1}$ of nested dyadic intervals, with $\lambda_n\in\Lambda_{j_n}$ and $j_{n+1}>j_n$, such that 
\[
|c_{\lambda_n}|\geq 2^{-\left(h_{\min}+\frac{1}{n}\right)j_n}.
\]
Let
$
\{x\}=\bigcap_{n\geq1}\overline{\lambda_n}.
$
Since $\lambda_n\subset 3\lambda_{j_n}(x)$, for every $p<p_0$,
\[
l_{\lambda_{j_n}(x)}^{(p)}
\geq |c_{\lambda_n}|
\geq 2^{-\left(h_{\min}+\frac{1}{n}\right)j_n},
\]
and therefore
$
h_f^{(p)}(x)\leq h_{\min}.
$
Since $h_{\min}$ is the minimal possible $p$-regularity, it gives the conclusion.
\end{proof}

Let us conclude this section with the proof that, almost surely, for every $p<p_0$ and every \mbox{$h\in\ocint{h_{\min},h_{\max}^{(p)}}$},
\begin{equation}\label{aim_inf}
\mathscr{D}_f^{(p)}(h)\geq \left(h+\frac{1}{p}\right)\sup_{\alpha\in\ocint{\frac{-1}{p},h}}\frac{\bm{\nu}(\alpha)}{\alpha+\frac{1}{p}},
\end{equation}
which suffices since $\bm{\nu}\geq\nu_{\vec{c}}$. We first establish in Lemma~\ref{gauge} that the proof reduces  to finding a suitable gauge function for the set of points whose $p$-exponent is at most $h$. To this end, we need to ensure that Theorem~\ref{mainSup} also holds for the increasing $p$-spectrum. This is the purpose of Corollary~\ref{spectre_croissant2}, which relies on the two following lemmas. The first is an adaptation of Proposition~\ref{spectre_densite}, and the second  can be proved similarly to Equation~\eqref{eq:enveloppe}.

\begin{lemma}\label{spectre_croissant}
For every 
$h\geq\frac{-1}{p}$, we have
\[
\dim_{\mathcal{H}}\Big\{x\in\ccint{0,1}:h_f^{(p)}(x)\leq h\Big\}\leq \nu_{\vec{c}}^{(p)}(h).
\]
\end{lemma}

\begin{lemma}\label{enveloppep}
For every $\alpha\in\R$,
\[
\nu_{\vec{c}}^{(p)}(\alpha)\leq\sup_{\alpha'\leq \alpha} 
\rho_{\vec{c}}^{(p)}(\alpha').
\]
\end{lemma}

\begin{cor}\label{spectre_croissant2}
For every 
$h\geq\frac{-1}{p}$, we have
\[
\dim_{\mathcal{H}}\Big\{x\in\ccint{0,1}:h_f^{(p)}(x)\leq h\Big\}\leq \Big(h+\frac{1}{p}\Big)\sup_{\alpha\in\ocint{\frac{-1}{p},h}}\frac{\nu_{\vec{c}}(\alpha)}{\alpha+\frac{1}{p}}\leq \Big(h+\frac{1}{p}\Big)\sup_{\alpha\in\ocint{\frac{-1}{p},h}}\frac{\bm{\nu}(\alpha)}{\alpha+\frac{1}{p}}.
\]
\end{cor}

\begin{proof}
This follows by applying in succession Lemma~\ref{spectre_croissant}, Lemma~\ref{enveloppep}, Proposition~\ref{class_restr}, Theorem~\ref{petitsoleil}, and Corollary~\ref{eg_nu}.
\end{proof}

\begin{lemma}\label{gauge}
Let $p<p_0$ and $h\in\ccint{h_{\min},h_{\max}^{(p)}}$. If there exists a gauge function $\xi$ satisfying
\[
\mathcal{H}^\xi\Big(\Big\{x\in\ccint{0,1}:h_f^{(p)}(x)\leq h\Big\}\Big)>0\quad\text{and}\quad \lim_{r\rightarrow0^+}\frac{\log\xi(r)}{\log r}\geq \Big(h+\frac{1}{p}\Big)\sup_{\alpha\in\ocint{\frac{-1}{p},h}}\frac{\bm{\nu}(\alpha)}{\alpha+\frac{1}{p}},
\]
then Equation~\eqref{aim_inf} is satisfied.
\end{lemma}

\begin{proof}
We can write
\[
\Big\{x\in\ccint{0,1}:h_f^{(p)}(x)=h\Big\}
=\Big\{x\in\ccint{0,1}:h_f^{(p)}(x)\leq h\Big\}\setminus\bigcup_{n\in\N}\Big\{x\in\ccint{0,1}:h_f^{(p)}(x)\leq h-\frac{1}{n}\Big\}
\]
and define
\[
D_h=\Big(h+\frac{1}{p}\Big)\sup_{\alpha\in\ocint{\frac{-1}{p},h}}\frac{\bm{\nu}(\alpha)}{\alpha+\frac{1}{p}}.
\]
Clearly, for every $n\in\N$,
there exists $\varepsilon_n>0$ such that 
$
\xi(r)<r^{D_{h-\frac{1}{n}}+\varepsilon_n}
$
when $r$ is small enough. It follows that
\[
\mathcal{H}^\xi\Big(\Big\{x\in\ccint{0,1}:h_f^{(p)}(x)\leq h-\frac{1}{n}\Big\}\Big)=0,
\]
because otherwise we would have
\[
\dim_{\mathcal{H}}\Big\{x\in\ccint{0,1}:h_f^{(p)}(x)\leq h-\frac{1}{n}\Big\}\geq D_{h-\frac{1}{n}}+\varepsilon_n>D_{h-\frac{1}{n}}
\]
which contradicts Corollary~\ref{spectre_croissant2}. Finally, we obtain
\[
\mathcal{H}^\xi\Big(\Big\{x\in\ccint{0,1}:h_f^{(p)}(x)=h\Big\}\Big)
>0,
\]
which implies Equation~\eqref{aim_inf}.
\end{proof}

To construct this gauge function, we will rely on the following result, which corresponds to a simplified version of the general mass transference principle stated in \cite[Theorem 2.2]{edouard}. In our setting, the theorem is applied to the Lebesgue measure on $[0,1]$, which allows for a more straightforward formulation. Note that for every ball $B=B(x,r)$ in $\R^d$ and every $a>0$, $B^a$ stands for the ball centered in $x$ and of radius $r^a$, i.e. $B^a=B(x,r^a)$.

\begin{thm}\label{edouard}
Let $(B_n)_{n\in\N}$ be a sequence of balls of $\ccint{0,1}$ and $(\gamma_n)_{n\in\N}\in\coint{1,+\infty}^{\N}$ a sequence of contracting ratios. Let 
\[
s=\sup\Bigg\{\frac{1}{\gamma}:\leb\Big(\limsup_{k\,:\,\gamma_k\leq \gamma}B_k\Big)=1\Bigg\}.
\]
Then there exists a gauge function $\xi:\coint{0,+\infty}\rightarrow\coint{0,+\infty}$ such that
\[
\lim_{r\rightarrow0^+}\frac{\log\xi(r)}{\log r}=s \quad\text{and}\quad \mathcal{H}^\xi\Big(\limsup_{n\rightarrow+\infty}B_n^{\gamma_n}\Big)>0.
\]
\end{thm}

The classical Mass Transference Principle of Beresnevich and Velani could provide the corresponding lower bound for the Hausdorff dimension of the sublevel set
$\{x\in\ccint{0,1}:h_f^{(p)}(x)\leq h\}$. Here, however, we need a positive Hausdorff measure with respect to a suitable gauge function in order to isolate the exact level set $\{x\in\ccint{0,1}:h_f^{(p)}(x)=h\}$, which is precisely provided by Theorem~\ref{edouard}. We therefore seek a limsup set to which Theorem~\ref{edouard} applies and which is contained in $\{x\in\ccint{0,1}:h_f^{(p)}(x)\leq h\}$.
Corollary~\ref{E_exp_2} does not directly provide such a set, so a modification is required.
Consider $(\alpha_n)_{n\in\N}$ a sequence whose elements belong to $\ooint{h_{\min},+\infty}\setminus\mathcal{D}$ and which is dense in $\coint{h_{\min},+\infty}$. For every $h\geq h_{\min}$ and every $p<p_0$,
we set
\begin{equation}\label{eq:E}
  E_h^{(p)}=\limsup_{j\rightarrow+\infty}\bigcup_{n\leq j\,:\,\alpha_n\leq h}\;\bigcup_{k\in F_j(\alpha_n)}B\Big(k2^{-j},2^{\delta_n^{(p)}\big(-j+\frac{4}{\bm{\nu}(\alpha_n)}\log_2j\big)}\Big),  
\end{equation}
where
\[
\delta_n^{(p)}=\frac{\alpha_n+\frac{1}{p}}{h+\frac{1}{p}}.
\]
Note that the condition $n\leq j$ ensures that, at each scale $j$, a finite number of balls are taken into account in the definition of $E_h^{(p)}$, and therefore that it is a limsup set over $j$.

\begin{rem}\label{rem:h_min}
In the case $\bm{\nu}(h_{\min})>0$, we include $h_{\min}$ in the sequence $(\alpha_n)_{n\in\N}$. Consequently, Theorem~\ref{inf_spectre} also holds at $h = h_{\min}$, so that Proposition~\ref{palier} is encompassed by Theorem~\ref{inf_spectre}.
\end{rem}

\begin{prop}\label{intermed}
For every $h\geq h_{\min}$ and every $p<p_0$, 
\[
E^{(p)}_h\subseteq\Big\{x\in\ccint{0,1}:h^{(p)}_f(x)\leq h\Big\}.
\]
\end{prop}

\begin{proof}
Let $h\geq h_{\min}$, $p<p_0$, $x\in E_h^{(p)}$ and $\delta>0$ be such that $\eta_f(p)>\delta$. By definition of $E_h^{(p)}$, there exists a sequence $(j_m)_{m\in\N}$ such that for every $m\in\N$, one can find $\alpha(j_m)\leq h$ and $k(j_m)\in\Lambda_{j_m}$ satisfying 
\[
\modu{x-k(j_m)2^{-j_m}}<2^{\delta(j_m)\left(-j_m+\frac{4}{\bm{\nu}(\alpha(j_m))}\log_2j_m\right)}\quad\text{and}\quad \modu{c_{j_m,k(j_m)}}\geq 2^{-\alpha(j_m)j_m},
\]
with
$$
\delta(j_m)=\frac{\alpha(j_m)+\frac{1}{p}}{h+\frac{1}{p}}.
$$
Moreover, there exists $J\in\N$ such that for every $j\geq J$ and every $\lambda\in\Lambda_j$,
$
\modu{c_\lambda}<2^{\frac{(1-\delta)j}{p}}.
$
Together, these estimates imply that, for every $m\in\N$ such that $j_m\geq J$,
$$
\delta(j_m)>\frac{\delta}{hp+1}.
$$
Hence, the sequence $(\delta(j_m))_{n \in \mathbb{N}}$ is bounded from below by a strictly positive constant. Since it is also bounded from above by $1$, we can, up to extraction of a subsequence, assume that it converges to some $l \geq \frac{\delta}{h p + 1} > 0$. Proceeding as in Lemma~\ref{E_exp}, for each $\varepsilon > 0$ and each $m \in \mathbb{N}$, we define
\[
j'_m=\Big\lfloor\delta(j_m) \Big(j_m-\frac{4}{\bm{\nu}(\alpha(j_m))+\varepsilon}\log_2j_m\Big)\Big\rfloor,
\]
so that
$\lambda_{j_m,k(j_m)}\subseteq 3\lambda_{j'_m}(x)$ and, if $m$ is large enough,
\[
l_{\lambda_{j'_m}(x)}^{(p)}\geq 2^{-\alpha(j_m) j_n}2^{-\frac{j_m-j'_m}{p}}\geq 2^{-\big(h+\frac{1}{p}\big)\frac{\delta(j_m)}{\delta(j_m)(1-\varepsilon)-\varepsilon}j'_m}2^{\frac{j'_m}{p}}.
\]
Since $j'_m\rightarrow+\infty$ when $m\rightarrow+\infty$, we deduce that
$
h_f^{(p)}(x)\leq\big(h+\frac{1}{p}\big)\frac{l}{l\,(1-\varepsilon)-\varepsilon}-\frac{1}{p},
$
and the conclusion follows.
\end{proof}

We are finally able to prove the last expected result.

\begin{thm}\label{inf_spectre}
Almost surely, for all $p<p_0$ and all $h\in(h_{\min},h_{\max}^{(p)}]$, 
\[
\mathscr{D}^{(p)}_f(h)\geq \Big(h+\frac{1}{p}\Big)\sup_{\alpha\in\ocint{\frac{-1}{p},h}}\frac{\bm{\nu}(\alpha)}{\alpha+\frac{1}{p}}.
\]
\end{thm}

\begin{proof}
Using Lemma \ref{gauge} and Proposition \ref{intermed}, the proof boils down to showing that, almost surely, for every $p<p_0$ and every $h\in(h_{\min},h_{\max}^{(p)}]$, there exists a gauge function $\xi$ such that
\[
\mathcal{H}^\xi\big(E_h^{(p)}\big)>0\quad\text{and}\quad \lim_{r\rightarrow0^+}\frac{\log\xi(r)}{\log r}\geq \Big(h+\frac{1}{p}\Big)\sup_{\alpha\in\ccint{h_{\min},h}}\frac{\bm{\nu}(\alpha)}{\alpha+\frac{1}{p}}.
\]
Fix $p<p_0$ and $h\in(h_{\min},h_{\max}^{(p)}]$. These values being fixed, in order not to overcomplicate the notations, we drop the indices. Recall first that $E$, defined from Equation~\eqref{eq:E}, can be viewed as a $\limsup$ set of balls
\[
B_{j,n,k}=B\Big(k2^{-j},2^{\delta_n^{(p)}\big(-j+\frac{4}{\bm{\nu}(\alpha_n)}\log_2j\big)}\Big)
\]
with $n\leq j$, $\alpha_n\leq h$ and $k\in F_j(\alpha_n)$. Now, for every such $n\in\N$, choose $\varepsilon_n>0$ such that $2\varepsilon_n<\bm{\nu}(\alpha_{n})$, and define
\[
\beta_n=\Big(h+\frac{1}{p}\Big)\frac{\bm{\nu}(\alpha_{n})-\varepsilon_n}{\alpha_{n}+\frac{1}{p}}.
\]
Notice that
\[
0<\Big(h_{\min}+\frac{1}{p}\Big)\frac{\bm{\nu}(\alpha_{n})-\varepsilon_n}{h_{\max}+\frac{1}{p}}\leq\beta_n\leq\Big(h_{\max}+\frac{1}{p}\Big)\frac{\bm{\nu}(\alpha_{n})}{\alpha_{n}+\frac{1}{p}}\leq1.
\]
It follows that $\gamma_n=\frac{1}{\beta_n}$ is well-defined, larger or equal to 1, and satisfies
\[
B_{j,n,k}\supseteq B\Big(k2^{-j},2^{-(\bm{\nu}(\alpha_{n})-\varepsilon_n)j+2\log_2j}\Big)^{\gamma_n}
\]
for every $n\in\N$. For every $(j,n,k)$, we set $\gamma_{j,n,k}=\gamma_n$. We only need to construct a full probability event $\Omega^\ast$ independent of $h$ and $p$, on which
\small\[
s\coloneqq\sup\Big\{\frac{1}{\gamma}:\leb\Big(\limsup_{(j,n,k)\,:\,\gamma_{j,n,k}\leq\gamma}B\Big(k2^{-j},2^{-(\bm{\nu}(\alpha_{n})-\varepsilon_n)j+2\log_2j}\Big)\Big)=1\Big\}\geq\Big(h+\frac{1}{p}\Big)\sup_{\alpha\in\ccint{h_{\min},h}}\frac{\bm{\nu}(\alpha)}{\alpha+\frac{1}{p}}.
\]\normalsize
Notice that
\begin{align*}
s 
& \geq\sup\Big\{\beta_n:\leb\Big(\limsup_{j\rightarrow+\infty}\bigcup_{k\in F_j(\alpha_{n})}B\left(k2^{-j},2^{-(\bm{\nu}(\alpha_{n})-\varepsilon_n)j+2\log_2j}\right)\Big)=1\Big\} \\
& =\sup\big\{\beta_n:\leb(E(\alpha_{n},\bm{\nu}(\alpha_{n})-\varepsilon_n))=1\big\}.
\end{align*}
In view of Proposition \ref{leb1}, there exists a full probability event $\Omega^\ast$ independent of $h$ and $p$ such that for every $n\in\N$ and every $\varepsilon>0$ satisfying $\bm{\nu}(\alpha_n)>\varepsilon$, one has $\leb(E(\alpha_n,\bm{\nu}(\alpha_n)-\varepsilon))=1.$ It follows that, on $\Omega^\ast$, 
\[
s\geq \Big(h+\frac{1}{p}\Big)\sup_{n\in\N}\frac{\bm{\nu}(\alpha_{n})-\varepsilon_n}{\alpha_{n}+\frac{1}{p}}=\Big(h+\frac{1}{p}\Big)\sup_{\alpha\in\ccint{h_{\min},h}}\frac{\bm{\nu}(\alpha)}{\alpha+\frac{1}{p}}
\]
as expected.
\end{proof}

\begin{cor}\label{cor:RWS_pLD}
Let $f$ be a Random Wavelet Series and let $0<p<p_0$. Almost surely, for every
$h\in[h_{\min},h_{\max}^{(p)}]$,
\[
D_f^{(p)}(h)=\rho_{\vec c}^{(p)}(h)
=\rho_{\vec c}^{(p),*}(h)
=\Big(h+\frac1p\Big) \sup_{\alpha\in(-1/p,h]} \frac{\bm{\rho}(\alpha)}
{\alpha+\frac1p} 
\]
In particular, the $p$-large deviation spectrum of the $p$-leaders coincides almost surely with the $p$-spectrum of $f$.
\end{cor}

\begin{proof}
By Propositions~\ref{spectre_densite}, \ref{class_restr}, and~\ref{upperbound_RWS}, almost surely,
$$
\mathscr{D}_f^{(p)}(h)\leq\rho_{\vec c}^{(p)}(h)
\leq\rho_{\vec c}^{(p),*}(h)\leq D_f^{(p)}(h).
$$
On the other hand, Theorem~\ref{main_RWS} gives
$\mathscr{D}_f^{(p)}(h)=D_f^{(p)}(h)$ for every
$h\in[h_{\min},h_{\max}^{(p)}]$. Hence all the previous inequalities are equalities. Finally, by the definition of $D_f^{(p)}$ and Proposition~\ref{propW},
$$
D_f^{(p)}(h)
=
\Big(h+\frac1p\Big)\sup_{\alpha\in(-1/p,h]}
\frac{\bm{\rho}(\alpha)}{\alpha+\frac1p},
$$
which concludes the proof.
\end{proof}

\subsection{The $p$-scaling function of Random Wavelet Series}

We conclude this section by deriving the $p$-scaling function of a Random Wavelet Series. This provides a complementary formulation of the previous results which is particularly relevant from the point of view of numerical multifractal analysis.

For $q>0$ and $j\in\N$, define the $p$-leader structure function by
\[
S_f^{(p)}(q,j) = 2^{-j}\sum_{\lambda\in\Lambda_j} \big(l_\lambda^{(p)}\big)^q
\]
and the corresponding $p$-scaling function by
\[
\eta_f^{(p)}(q)
=
\liminf_{j\to+\infty} 
\frac{\log S_f^{(p)}(q,j)}
{\log 2^{-j}}.
\]
Corollary~\ref{cor:RWS_pLD} allows us to compute this function explicitly from the theoretical wavelet profile.

\begin{prop}\label{prop:p_scaling_RWS}
Let $f$ be a Random Wavelet Series and let $0<p<p_0$. Almost surely, for every $q>0$, the $p$-scaling function of $f$ is given by
$$
\eta_f^{(p)}(q)=
\begin{cases}
q h_{\max}^{(p)}& \text{ if }
0<q\leq \displaystyle\frac{1}{h_{\max}^{(p)}+\frac{1}{p}},\\[2ex]
\displaystyle\inf_{\alpha\geq h_{\min}}
\left(1+q\alpha-\bm{\rho}(\alpha)\right)& \text{ if }
q\geq \displaystyle\frac{1}{h_{\max}^{(p)}+\frac{1}{p}}.
\end{cases}
$$
\end{prop}

\begin{proof}
For $q>0$, the $p$-scaling function can be expressed in terms of the large-deviation spectrum of the $p$-leaders via
\begin{equation}\label{eq:eta_f}
    \eta_f^{(p)}(q)=\inf_{h\geq h_{\min}}\big(1+qh-\rho_{\vec c}^{(p)}(h)\big),
\end{equation}
as done in~\cite{Esser17}. Since $\rho_{\vec c}^{(p)}(h)\leq1$ for every $h$ and, by Corollary~\ref{cor:RWS_pLD}, $\rho_{\vec c}^{(p)}(h_{\max}^{(p)})=1$, the values $h>h_{\max}^{(p)}$ cannot decrease the infimum. Hence
$$
\eta_f^{(p)}(q)=\inf_{h\in[h_{\min},h_{\max}^{(p)}]}\big(1+qh-\rho_{\vec c}^{(p)}(h)\big).
$$
Using Corollary~\ref{cor:RWS_pLD}, we obtain
\begin{align*}
\eta_f^{(p)}(q)&=\inf_{h\in[h_{\min},h_{\max}^{(p)}]}\inf_{\alpha\in\ccint{h_{\min},h}}\left(1+qh-\left(h+\frac{1}{p}\right)\frac{\bm{\rho}(\alpha)}{\alpha+\frac{1}{p}}\right)\\
& =\inf_{\alpha\in[h_{\min},h_{\max}^{(p)}]}\inf_{h\in[\alpha,h_{\max}^{(p)}]}\left(1+qh-\left(h+\frac{1}{p}\right)\frac{\bm{\rho}(\alpha)}{\alpha+\frac{1}{p}}\right).
\end{align*}
For a fixed $\alpha$, the expression inside the infimum is affine in $h$, with slope
$
q-\frac{\bm{\rho}(\alpha)}{\alpha+\frac{1}{p}}$.
By the definition of $h_{\max}^{(p)}$ and Corollary~\ref{cor:RWS_pLD},
$$
\sup_{\alpha\in[h_{\min},h_{\max}^{(p)}]}\frac{\bm{\rho}(\alpha)}{\alpha+\frac{1}{p}}
=\frac{1}{h_{\max}^{(p)}+\frac{1}{p}}.
$$

Assume first that
$
q\geq\frac{1}{h_{\max}^{(p)}+\frac{1}{p}}.
$
Then the slope is non-negative for every $\alpha\in[h_{\min},h_{\max}^{(p)}]$, so that
$$
\eta_f^{(p)}(q)=\inf_{\alpha\in\ccint{h_{\min},h_{\max}^{(p)}}}\big(1+q\alpha-\bm{\rho}(\alpha)\big).
$$
This infimum can be extended to all $\alpha\geq h_{\min}$. Indeed, if $\alpha>h_{\max}^{(p)}$, then, since $\bm{\rho}(\alpha)\leq1$,
$
1+q\alpha-\bm{\rho}(\alpha)\geq q\alpha>q h_{\max}^{(p)}$,
and taking $h=h_{\max}^{(p)}$ in the infimum in \eqref{eq:eta_f} gives $\eta_f^{(p)}(q)\leq qh_{\max}^{(p)}$. Hence
$$
\eta_f^{(p)}(q)=\inf_{\alpha\geq h_{\min}}\big(1+q\alpha-\bm{\rho}(\alpha)\big).
$$

Consider  now the case
$
0<q\leq\frac{1}{h_{\max}^{(p)}+\frac{1}{p}}.
$
For every $h\in[h_{\min},h_{\max}^{(p)}]$, Corollary~\ref{cor:RWS_pLD} yields
$$
\rho_{\vec c}^{(p)}(h)\leq\Big(h+\frac{1}{p}\Big)\frac{1}{h_{\max}^{(p)}+\frac{1}{p}}.
$$
Therefore
$$
1+qh-\rho_{\vec c}^{(p)}(h)\geq1+qh-\frac{h+\frac{1}{p}}{h_{\max}^{(p)}+\frac{1}{p}}.
$$
The minimum of the right-hand side on $[h_{\min},h_{\max}^{(p)}]$ is reached at $h=h_{\max}^{(p)}$. Consequently,
$$
1+qh-\rho_{\vec c}^{(p)}(h)\geq qh_{\max}^{(p)}.
$$
On the other hand, since $\rho_{\vec c}^{(p)}(h_{\max}^{(p)})=1$,
$$
\eta_f^{(p)}(q)\leq1+qh_{\max}^{(p)}-\rho_{\vec c}^{(p)}(h_{\max}^{(p)})=qh_{\max}^{(p)}.
$$
Hence
$$
\eta_f^{(p)}(q)=qh_{\max}^{(p)}.
$$
This concludes the proof.
\end{proof}

\begin{ex}\label{ex:two_regularities}
Let $-\frac{1}{p}<\alpha_1<\alpha_2$ and $0<\eta_1,\eta_2<1$, and consider a Random Wavelet Series whose theoretical wavelet density is given by
$$
\bm{\rho}(\alpha)=
\begin{cases}
\eta_1, & \text{if }\alpha=\alpha_1,\\
\eta_2, & \text{if }\alpha=\alpha_2,\\
-\infty, & \text{otherwise}.
\end{cases}
$$
In other words, at scale $j$, the series contains asymptotically $2^{\eta_1j}$ coefficients of order $2^{-\alpha_1j}$ and $2^{\eta_2j}$ coefficients of order $2^{-\alpha_2j}$. In this case,
$$
\frac{1}{h_{\max}^{(p)}+\frac{1}{p}}
=
\max\left(
\frac{\eta_1}{\alpha_1+\frac{1}{p}},
\frac{\eta_2}{\alpha_2+\frac{1}{p}}
\right).
$$
Moreover, for $q\geq \frac1{h_{\max}^{(p)}+\frac1p}$, Proposition~\ref{prop:p_scaling_RWS} gives
$$
\eta_f^{(p)}(q)=\min\left(1-\eta_1+q\alpha_1,1-\eta_2+q\alpha_2\right).
$$
Assume first that $\eta_1<\eta_2$. The two affine functions intersect at the positive value
$
q_0=\frac{\eta_2-\eta_1}{\alpha_2-\alpha_1}.
$
If
$$
\frac{\eta_2}{\alpha_2+\frac{1}{p}}
>
\frac{\eta_1}{\alpha_1+\frac{1}{p}},
$$
then
$$
\frac{1}{h_{\max}^{(p)}+\frac{1}{p}}
=
\frac{\eta_2}{\alpha_2+\frac{1}{p}}
< q_0,
$$
and therefore 
$$
\eta_f^{(p)}(q)=
\begin{cases}
q \left(\frac{\alpha_2+\frac{1}{p}}{\eta_2}-\frac{1}{p}\right) & \text{ if } 0<q\leq\displaystyle\frac{\eta_2}{\alpha_2+\frac{1}{p}},\\[2ex]
1-\eta_2+q\alpha_2& \text{ if } \displaystyle\frac{\eta_2}{\alpha_2+\frac{1}{p}}\leq q\leq q_0,\\[2ex]
1-\eta_1+q\alpha_1 & \text{ if } q\geq q_0
\end{cases}
$$
see Figure \ref{figure_pscaling}. Thus both populations are visible in the $p$-scaling function: the coefficients of order $\alpha_2$ determine the intermediate regime, whereas the coefficients of smaller order $\alpha_1$ dominate the large moments.
If, on the other hand,
$$
\frac{\eta_1}{\alpha_1+\frac{1}{p}}
\geq
\frac{\eta_2}{\alpha_2+\frac{1}{p}},
$$
then
$
q_0\leq\frac{1}{h_{\max}^{(p)}+\frac{1}{p}} = \frac{\eta_1}{\alpha_1+\frac{1}{p}},
$
so that the intermediate regime disappears. In this case,
$$
\eta_f^{(p)}(q)=
\begin{cases}
q \left(\frac{\alpha_1+\frac{1}{p}}{\eta_1}-\frac{1}{p}\right) & \text{ if } 0<q\leq\displaystyle\frac{\eta_1}{\alpha_1+\frac{1}{p}},\\[2ex]
1-\eta_1+q\alpha_1 & \text{ if } q\geq\displaystyle\frac{\eta_1}{\alpha_1+\frac{1}{p}}.
\end{cases}
$$
Finally, assume that $\eta_1\geq\eta_2$. Since $\alpha_1<\alpha_2$, we necessarily have
$$
\frac{\eta_1}{\alpha_1+\frac{1}{p}}
>
\frac{\eta_2}{\alpha_2+\frac{1}{p}},
$$
and
$$
q_0=\frac{\eta_2-\eta_1}{\alpha_2-\alpha_1}\leq0.
$$
Hence, for every $q>0$,
$$
1-\eta_1+q\alpha_1<1-\eta_2+q\alpha_2,
$$
and the second population never contributes to the $p$-scaling function. We therefore obtain
$$
\eta_f^{(p)}(q)=
\begin{cases}
q \left(\frac{\alpha_1+\frac{1}{p}}{\eta_1}-\frac{1}{p}\right) & \text{ if } 0<q\leq\displaystyle\frac{\eta_1}{\alpha_1+\frac{1}{p}},\\[2ex]
1-\eta_1+q\alpha_1 & \text{ if } q\geq\displaystyle\frac{\eta_1}{\alpha_1+\frac{1}{p}}.
\end{cases}
$$
Thus, although both populations are present in the theoretical wavelet density, the second one is invisible in the $p$-scaling function for positive moments.
\end{ex}

\begin{figure}[ht]
    \centering
\begin{tikzpicture}[scale=0.9]
\begin{axis}[
    width=12cm, height=8cm,
    axis lines=middle,
    xtick=\empty,
    ytick=\empty,
    xmin=-0.1, xmax=2.2,
    ymin=-0.15, ymax=1.05,
    x=4.8cm,
    y=7cm
]

% p-scaling function
\addplot[red, thick, domain=0:0.8421] {0.6875*x};
\addplot[red, thick, domain=0.8421:1.2857] {0.2 + 0.45*x};
\addplot[red, thick, domain=1.2857:2.1] {0.65 + 0.1*x};

% Function obtained directly from the wavelet density
\addplot[blue, dashed, domain=0:1.2857] {0.2 + 0.45*x};
\addplot[blue, dashed, domain=1.2857:2.1] {0.65 + 0.1*x};

% Critical values
\draw[loosely dashed] (0.8421,0)--(0.8421,0.5789);
\draw[loosely dashed] (1.2857,0)--(1.2857,0.7786);

\draw (0.8421,-0.01)--(0.8421,0.01);
\draw (0.8421,0) node[below] {\scriptsize $\displaystyle\frac{1}{h_{\max}^{(p)}+\frac{1}{p}}$};

\draw (1.2857,-0.01)--(1.2857,0.01);
\draw (1.2857,0) node[below] {\footnotesize $q_{0}$};

% Labels of the three affine regimes
\node[red] at (0.24,0.25) {\footnotesize $qh_{\max}^{(p)}$};
\node[red] at (1.0,0.78) {\footnotesize $1-\eta_2+q\alpha_2$};
\node[red] at (1.73,0.88) {\footnotesize $1-\eta_1+q\alpha_1$};

\end{axis}
\end{tikzpicture}

\caption{The $p$-scaling function of the Random Wavelet Series of Example~\ref{ex:two_regularities} in the case
$\frac{1}{h_{\max}^{(p)}+1/p}=\frac{\eta_2}{\alpha_2+\frac{1}{p}}
 <q_{0}$. The function
$q\mapsto\inf_{\alpha\geq h_{\min}}(1+q\alpha-\bm{\rho}(\alpha))$
is represented in blue and the $p$-scaling function $\eta_f^{(p)}$ in red.} 
\label{figure_pscaling}
\end{figure}

\section{Prevalent $p$-spectrum in $S^\nu$ spaces}\label{secSnutot}

In this section, we prove the generic optimality of the $p$-large deviation wavelet formalism, that is, Theorem~\ref{mainPrev}, within the spaces $S^\nu$. To this end, we begin by recalling the notion of prevalence, as well as the definition of the spaces $S^\nu$.

\subsection{Prevalence and $S^\nu$ spaces}\label{secSnu}

The notion of prevalence is intended to describe which sets may be considered as \textit{large} in a measure-theoretic sense. In $\R^n$, a set is typically called \textit{small} if its Lebesgue measure is null. However, the only locally finite and translation-invariant measure defined on the Borel subsets of an infinite dimensional Banach space is the trivial measure. The notion of \textit{prevalence} was independently introduced by Christensen (\cite{christensen}) and Hunt, Sauer and Yorke (\cite{hunt}) in order to compensate for the lack of such a measure. More precisely, it naturally generalizes the class of null Lebesgue measure sets without the use of a particular measure. 

\begin{defi}
A non-trivial measure $\mu$ defined on the Borel subset of a 
Polish space $X$ is said to be \textit{transverse} to a Borel subset $B$ of $X$ if $\mu(B+x)=0$ for every $x\in X$.
Furthermore, a Borel subset $B$ of $X$ is said to be \textit{shy} if there exists a measure that is transverse to $B$. More generally, a subset of $X$ is \textit{shy} if it can be included in a shy Borel subset of $X$. Moreover, a subset of $X$ is said to be \textit{prevalent} if its complement is shy.
\end{defi}

Let us now recall the definition and some properties of $S^\nu$ spaces introduced in \cite{jaffardBeyond} (see also \cite{aubryTopological}). We consider an \emph{admissible profile} $\nu$, i.e. a function
$
\nu:\R\rightarrow\{-\infty\}\cup\ccint{0,1}
$
which is non-decreasing, right-continuous and satisfies
\[
\alpha_{\min}\coloneqq\inf\{\alpha\in\R:\nu(\alpha)\geq 0\}\in\R.
\]
In this case, $\nu(\alpha)=-\infty$ for every $\alpha<\alpha_{\min}$ and $\nu(\alpha)\geq 0$ for every $\alpha\geq \alpha_{\min}$.

\begin{defi}
The space $S^\nu$ is the set of functions $f$ whose sequence of wavelet coefficients $\vec{c}$ satisfies the following property: for every $\alpha \in \mathbb{R}$, every $\varepsilon > 0$, and every $C > 0$, there exists $J \in \mathbb{N}$ such that
\[
\#\left\{k\in\{0,\ldots,2^j-1\}:\modu{c_{j,k}}\geq C2^{-\alpha j}\right\}\leq 2^{(\nu(\alpha)+\varepsilon)j}\, , \quad\forall j\geq J.
\]
\end{defi}

In other words, $S^\nu$ is the space of functions $f$ whose wavelet coefficient sequence $\vec{c}$ satisfies $\nu_{\vec{c}}(\alpha) \leq \nu(\alpha)$ for every $\alpha \in \mathbb{R}$. This space can be shown to be robust (i.e., independent of the choice of a regular wavelet basis used to compute the coefficients), vectorial, metric, complete, and separable \cite{aubryTopological}.
Hence, it is suitable for the study of generic properties. Moreover, it is known from \cite{aubryPrevalence} that the set of sequences $\vec{c} \in S^\nu$ for which $\nu_{\vec{c}} = \nu$ is prevalent.

\subsection{Proof of Theorem~\ref{mainPrev}}\label{secpSnu}

We fix an admissible profile $\nu$ such that $\nu(\alpha) > 0$ for all $\alpha > \alpha_{\min}$.
 We set
\begin{equation}\label{eq:pnu}
p_\nu=\inf_{\alpha\in\coint{\alpha_{\min},0}}\frac{\nu(\alpha)-1}{\alpha},
\end{equation}
and we assume $\nu(0)<1$ if $\alpha_{\min}\leq 0$. 
In this case, using the classical property
$
\eta_f(p)=\sup\big\{s\in\R:\vec{c}\in b_{p,\infty}^{\frac{s}{p}}\big\} 
$
(see \cite{Jaffard1997}) together with 
\[
S^\nu\subseteq \bigcap_{\varepsilon>0}b_{p,\infty}^{\frac{\eta(p)}{p}-\varepsilon},\quad\text{with}\quad
\eta(p)=\inf_{\alpha\geq\alpha_{\min}}(\alpha p-\nu(\alpha)+1)
\]
(see \cite{aubryTopological}), it can be shown that $\eta_f(p)>0$ for all $p<p_\nu$ and all $f\in S^\nu$, as required.

\medskip

To establish Theorem~\ref{mainPrev}, we rely on the fact that a property $\mathcal{P}$ in a Polish space $E$ is prevalent if one can construct a process $X$ which has almost surely its values in $E$ and such that $X+f$ satisfies $\mathcal{P}$ for all $f\in E$. Indeed, in this case, the distribution of $X$ is transverse to the set of functions in $E$ which do not satisfy $\mathcal{P}$, and its complement is therefore prevalent. Let us now construct such a process. It follows naturally from Section~\ref{secpRWS} to consider a specific type of Random Wavelet Series.

\begin{defi}
A RWS is said to be \textit{associated} to $\nu$ if 
\begin{itemize}
\item for every $\alpha\in\R$, one has
\[
\limsup_{j\rightarrow+\infty}\frac{1}{j}\log_2\left(2^j\bm{\rho_j}(\ocint{-\infty,\alpha})\right)=\nu(\alpha)
\]
i.e. $\nu=\bm{\nu}=\bm{\lambda}$,
\item $\nu(\alpha)\geq 0\Rightarrow2^j\bm{\rho_j}(\ocint{-\infty,\alpha})\geq j^2$ for every $j$ large enough. 
\end{itemize}
\end{defi}

The existence of a RWS associated to $\nu$ is established in \cite{aubryPrevalence}, and some of the following properties are mentioned.

\begin{prop}\label{RWSnu}
If $f$ is a RWS associated to $\nu$, then, almost surely, for every $p<p_\nu$,
\begin{enumerate}
\item one has
\[
h_{\max}^{(p)}
=
\inf\Bigg\{ h>-\frac{1}{p}: h+\frac{1}{p}
= \inf_{\alpha\in\ccint{\alpha_{\min},h}} \frac{\alpha+\frac{1}{p}}{\nu(\alpha)} \Bigg\};
\]
\item $f\in S^\nu$,
$\alpha_{\min}=h_{\min}$
and
$\nu_{\vec{c}}=\nu$;
\item for every
$h\in[-\frac1p,h_{\max}^{(p)}]$,
\[
\mathscr{D}_f^{(p)}(h)
\geq
\Big(h+\frac{1}{p}\Big)
\sup_{\alpha\in(-\frac{1}{p},h]}
\frac{\nu(\alpha)}{\alpha+\frac{1}{p}},
\]
while for every $h>h_{\max}^{(p)}$, $\mathscr{D}_f^{(p)}(h)=-\infty$. 
\end{enumerate}
\end{prop}

Together with the upper bound of Theorem~\ref{mainSup}, the previous
proposition shows that a RWS associated to $\nu$ almost surely satisfies
the $p$-large deviation wavelet formalism on $\coint{0,+\infty}$ for
every $p<p_\nu$.

\begin{proof}
The first item follows from the fact that $\bm{\nu}$ is the increasing hull of $\bm{\rho}$, together with the identity $\nu=\bm{\nu}$.
 Next, Corollary~\ref{eg_nu} ensures that $\nu_{\vec{c}}\leq\bm{\nu}=\nu$ and $\bm{\nu}(h_{\min})\geq0$, hence $f\in S^\nu$ and $h_{\min}\geq\alpha_{\min}$. Moreover, for all $\varepsilon>0$, $\bm{\nu}(\alpha_{\min}-\varepsilon)<0\leq \nu(\alpha_{\min})$, which implies that $\alpha_{\min}$ belongs to $W$, and therefore $\alpha_{\min}\geq h_{\min}$. This is enough to assert 
$\nu_{\vec{c}}=\nu$, using again Corollary~\ref{eg_nu}. Once this property is established, the third point follows directly from Theorem~\ref{main_RWS}.
\end{proof}

Choosing $X$ as a Random Wavelet Series associated with $\nu$ ensures that $X$ almost surely belongs to $S^\nu$ and has the required $p$-spectrum. To ensure that these properties are preserved after adding an arbitrary element of $S^\nu$, we define
\[
X=\sum_{j\in\N_0}\sum_{k=0}^{2^j-1}C_{j,k}\psi_{j,k},
\qquad
C_{j,k}=\varepsilon_{j,k}\modu{C_{j,k}},
\]
where the amplitudes $|C_{j,k}|$ are chosen so that $X$ is a RWS associated with $\nu$, and
$(\varepsilon_{j,k})_{j,k}$ is a family of i.i.d. Rademacher random variables, independent of the amplitudes. Then $X$ takes values in $S^\nu$, and for every fixed
\[
f=\sum_{j\in\N_0}\sum_{k=0}^{2^j-1}c_{j,k}\psi_{j,k}\in S^\nu,
\]
the process
\[
X+f=\sum_{j\in\N_0}\sum_{k=0}^{2^j-1}(C_{j,k}+c_{j,k})\psi_{j,k}
\]
also takes values in $S^\nu$. Although the random variables $C_{j,k}+c_{j,k}$ remain independent, they are no longer necessarily identically distributed, so $X+f$ is not, in general, a RWS.

It remains to prove that the $p$-spectrum of $X+f$ satisfies the formalism on $\coint{0,+\infty}$. For this, it is enough to establish an analogue of Lemma~\ref{j_alpha} for $X+f$.

\begin{lemma}
Let $\alpha\geq\frac{-1}{p}$ be such that $\bm{\lambda}(\alpha)>0$. Almost surely, for every $\varepsilon>0$ satisfying $\bm{\lambda}(\alpha)>\varepsilon$, for infinitely many scales $j$ and for all $\lambda\in\Lambda_j$, the smallest scale $J_\alpha(\lambda)\geq j$ for which there exists $\lambda'\in\Lambda_{J_\alpha(\lambda)}$ such that $\lambda'\subseteq \lambda$ and $\modu{C_{\lambda'}+c_{\lambda'}}\geq 2^{-\alpha J_\alpha(\lambda)}$ satisfies
\[
J_\alpha(\lambda)\leq \Big\lceil\frac{1}{\bm{\lambda}(\alpha)-\varepsilon}(j+2\log_2j)\Big\rceil.
\]
\end{lemma}

\begin{proof}
Fix $\varepsilon>0$ such that $\bm{\lambda}(\alpha)-\varepsilon>0$. As in Lemma~\ref{j_alpha}, we can choose sequences $(J_n)_{n\in\N}$ and $(j_n)_{n\in\N}$ such that, for every $n\in\N$,
\[
2^{J_n}\bm{\rho}_{J_n}(\ocint{-\infty,\alpha})
>
2^{(\bm{\lambda}(\alpha)-\varepsilon)J_n}
\]
and
\[
j_n=\max\left\{j\in\N:
\left\lceil
\frac{j+2\log_2j}{\bm{\lambda}(\alpha)-\varepsilon}
\right\rceil
\leq J_n
\right\}.
\]
For every $n\in\N$ and every $\lambda\in\Lambda_{j_n}$, let
\[
\Lambda_n(\lambda)
=
\{\lambda'\in\Lambda_{J_n}:\lambda'\subseteq\lambda\}.
\]
For every $\lambda'\in\Lambda_n(\lambda)$, define
\[
\sigma_{\lambda'}=
\begin{cases}
1,&\text{if }c_{\lambda'}\geq0,\\
-1,&\text{if }c_{\lambda'}<0,
\end{cases}
\]
and let
\[
\Lambda_n^+(\lambda)
=
\{\lambda'\in\Lambda_n(\lambda):
\varepsilon_{\lambda'}=\sigma_{\lambda'}\}.
\]
Thus, except for the case $c_{\lambda'}=0$, $\Lambda_n^+(\lambda)$ consists of the coefficients for which $C_{\lambda'}$ and $c_{\lambda'}$ have the same sign. In particular, for every $\lambda'\in\Lambda_n^+(\lambda)$,
\[
\modu{C_{\lambda'}+c_{\lambda'}}
=
\modu{C_{\lambda'}}+\modu{c_{\lambda'}}
\geq
\modu{C_{\lambda'}}.
\]
Moreover, since $\sigma_{\lambda'}$ is deterministic and the variables
$\varepsilon_{\lambda'}$ are independent Rademacher variables, each
$\lambda'$ belongs to $\Lambda_n^+(\lambda)$ with probability $\frac12$. Hence, 
 $\#\Lambda_n^+(\lambda)$ is a binomial random variable with parameters $2^{J_n-j_n}$ and $\frac12$. Let $B_{n,\lambda}$ denote the event  that $\#\Lambda_n^+(\lambda) \geq \frac{1}{3}2^{J_n-j_n}$. 
Therefore, by Hoeffding's inequality,
$$
\prob(B^c_{n,\lambda})
\leq \prob\Big(\mathbb{E}[\#\Lambda_n^+(\lambda)]-\#\Lambda_n^+(\lambda) \geq \frac16 2^{J_n-j_n}\Big)
\leq 2\exp\Big(-\frac{1}{18}2^{J_n-j_n}\Big)
\leq 2e^{-j_n},
$$
for $n$ large enough. Let $A_n$ be the event that there exists $\lambda\in\Lambda_{j_n}$ having no descendant $\lambda'\in\Lambda_{j'}$ of scale at most $J_n$ such that
$\modu{C_{\lambda'}+c_{\lambda'}}\geq2^{-\alpha j'}$.
For a fixed $\lambda\in\Lambda_{j_n}$, on the event $B_{n,\lambda}$ there are at least $\frac{1}{3}2^{J_n-j_n}$ intervals in $\Lambda_n^+(\lambda)$. Since $\Lambda_n^+(\lambda)$ depends only on the Rademacher signs, while the amplitudes $\modu{C_{\lambda'}}$ are independent of these signs, conditioning on $B_{n,\lambda}$ does not affect the distribution of the amplitudes. Hence
\[
\prob\left(
\forall\lambda'\in\Lambda_n^+(\lambda),
\modu{C_{\lambda'}}<2^{-\alpha J_n}
\,\middle|\,
B_{n,\lambda}
\right)
\leq
\left(
1-\bm{\rho}_{J_n}(\ocint{-\infty,\alpha})
\right)^{\frac{1}{3}2^{J_n-j_n}}.
\]
Moreover, for every $\lambda'\in\Lambda_n^+(\lambda)$,
\[
\modu{C_{\lambda'}+c_{\lambda'}}
\geq
\modu{C_{\lambda'}},
\]
so the same upper bound holds with
$\modu{C_{\lambda'}+c_{\lambda'}}<2^{-\alpha J_n}$ in place of
$\modu{C_{\lambda'}}<2^{-\alpha J_n}$. It follows that, conditioning by $B_{n,\lambda}$ and $B^c_{n,\lambda}$
\begin{align*}
\prob(A_n) & \leq \sum_{\lambda \in \Lambda_{j_n}} \prob(\forall\lambda'\in\Lambda_n^+(\lambda),
\modu{C_{\lambda'}+c_{\lambda'}}<2^{-\alpha J_n})\\
&\leq
2^{j_n}
\left(
1-\bm{\rho}_{J_n}(\ocint{-\infty,\alpha})
\right)^{\frac{1}{3}2^{J_n-j_n}}
+
2 \, 2^{j_n}e^{-j_n}\\
&\leq
2^{j_n}
\exp\left(
-\frac{1}{3}2^{J_n-j_n}
\bm{\rho}_{J_n}(\ocint{-\infty,\alpha})
\right)
+
2 \, 2^{j_n}e^{-j_n}
\end{align*}
By the definitions of $J_n$ and $j_n$,
\[
2^{J_n-j_n}\bm{\rho}_{J_n}(\ocint{-\infty,\alpha})
>
2^{(\bm{\lambda}(\alpha)-\varepsilon)J_n-j_n}
\geq j_n^2,
\]
so that
\[
\prob(A_n)
\leq
2^{j_n}e^{-j_n^2/3}+2 \, 2^{j_n}e^{-j_n}.
\]
The Borel--Cantelli lemma therefore yields the  conclusion for the fixed $\varepsilon$, and the result for every $\varepsilon$ follows by a countability argument.
\end{proof}

Once this lemma is established, the $p$-spectrum follows as in Section~\ref{secpRWS}. Note, however, that the lower bound thus provided is equivalently based on $\nu$, $\bm{\nu}$ or $\nu_{\vec{C}}$, but not on the profile $\nu_{\vec{C}+\vec{c}}$ of $X+f$, as would be required to ensure that $X+f$ satisfies the $p$-large deviation wavelet formalism. The prevalence of the set $\{f\in S^\nu:\nu_{\vec{c}}=\nu\}$ (stated in Section~\ref{secSnu}) is therefore required to conclude 
and to get Theorem \ref{mainPrev}, using the fact that the intersection of two prevalent sets is itself prevalent.

\begin{rem}
The argument above shows that the assumption that the coefficients are identically distributed at each scale can be partially relaxed, suggesting that similar results may hold in more general independent settings under suitable assumptions, such as uniform control of the tails of the coefficient distributions and sufficient concentration of the number of coefficients of a prescribed order. The independence assumption is of a different nature: in a Random Wavelet Series, it produces a form of statistical homogeneity in the spatial distribution of coefficients of a given order. Without independence, such a homogeneous distribution is no longer guaranteed, and correlations may create clustering or other spatial structures. As a consequence, the multifractal behaviour may differ substantially from that obtained in the independent setting. 
\end{rem}

\section{Proof of Proposition~\ref{prop:CounterEx}}\label{secCounterEx}

We conclude the paper by showing that the situation is fundamentally different for negative regularities. More precisely, we construct a function with a prescribed asymptotic distribution of wavelet coefficients whose $p$-spectrum at a negative exponent is strictly larger than the one predicted by the $p$-large deviation wavelet formalism. This proves that, contrary to the classical H\"older setting, Random Wavelet Series do not in general realize the maximal spectrum compatible with a prescribed wavelet coefficient distribution.

Let $p>0$ and $-\frac1p<h<\alpha<0$.
Choose $\rho\in\big(\frac{h-\alpha}{h},\,\alpha p+1\big)$.
This interval is non-empty since $\frac{h-\alpha}{h}<\alpha p+1$ if and only if $\alpha(hp+1)<0$, which follows from $\alpha<0$ and $h>-1/p$. We further choose $\varepsilon\in
\big(
0,\frac1p\big(\frac{h}{\alpha}-1\big)
\big)$
small enough so that
$\frac{h-\alpha}{h}+\varepsilon
<
\rho
<
\alpha p+1-\varepsilon$,
and fix $\delta\in(0,1/2)$. We now aim to construct a function $f$ satisfying
\begin{itemize}
    \item[(1)] $\rho_{\vec{c}}(\alpha)=\rho$,
    \item[(2)] $\rho_{\vec{c}}(\alpha')=-\infty$ for all $\alpha'\neq\alpha$,
    \item[(3)] $\eta_f(p)>0$,
    \item[(4)] $\mathscr{D}^{(p)}_f(h)\geq \frac{(\rho-1)h}{\alpha}+1$.
\end{itemize}
Indeed, once these conditions are satisfied, Theorem~\ref{mainSup} provides the matching upper bound for $\mathscr{D}^{(p)}_f(h)$, and Proposition~\ref{prop:CounterEx} follows.

To that end, let $(j_n)_{n\in\N}$ be a sequence of scales such that $j_{n+1}>\frac{h}{\alpha}j_n$
for every $n\in\N$, and set
\[
j_n'=\left\lceil\frac{h}{\alpha}j_n\right\rceil.
\]
For every $n\in\N$, we choose a subset
$\widetilde{\Lambda}_{j_n}\subset\Lambda_{j_n}$ consisting of
\[
N_n\coloneqq\big\lfloor 2^{\left((\rho-1)\frac{h}{\alpha}+1\right)j_n} \big\rfloor
\]
dyadic intervals distributed as uniformly as possible in $[0,1]$, so that
there are at most
$
\big\lceil\frac{2^{j_n}-N_n}{N_n-1}\big\rceil
$
intervals of $\Lambda_{j_n}\setminus\widetilde{\Lambda}_{j_n}$ between
two consecutive intervals of $\widetilde{\Lambda}_{j_n}$, as well as
before the first and after the last one. For every $n\in\N$ and every
$\lambda\in\widetilde{\Lambda}_{j_n}$, we then choose arbitrarily a subset
\[
\widetilde{\Lambda}_{j_n'}(\lambda)
\subset
\{\lambda'\in\Lambda_{j_n'}:\lambda'\subseteq\lambda\}
\]
consisting of
\[
\big\lfloor
2^{(\alpha p+1)j_n'-(hp+1)j_n}
\big\rfloor
\]
dyadic intervals. This choice is possible since
$
2^{(\alpha p+1)j_n'-(hp+1)j_n}
\leq 2^{j_n'-j_n},
$
which follows from $\alpha j_n'-hj_n\leq0$.
For every dyadic interval $\lambda'$, define
\[
c_{\lambda'}=
\begin{cases}
2^{-\alpha j_n'},
& \text{if }
\lambda'\in \widetilde{\Lambda}_{j_n'}(\lambda)
\ \text{for some }n\in\N \text{ and some } \lambda\in\widetilde{\Lambda}_{j_n} ,\\
0,
&\text{otherwise}.
\end{cases}
\]
We then define
\[
f=\sum_{j\in\N}\sum_{\lambda'\in\Lambda_j}c_{\lambda'}\psi_{\lambda'}.
\]

Let us now verify the required properties. Property~(2) follows immediately
from the construction. We first prove~(1). For every $n\in\N$,
\[
\#\big\{\lambda'\in\Lambda_{j_n'}:c_{\lambda'}=2^{-\alpha j_n'}\big\} =\sum_{\lambda\in\widetilde{\Lambda}_{j_n}}\#\widetilde{\Lambda}_{j_n'}(\lambda)=\big\lfloor 2^{\left((\rho-1)\frac{h}{\alpha}+1\right)j_n} \big\rfloor \, \big\lfloor 2^{(\alpha p+1)j_n'-(hp+1)j_n} \big\rfloor .
\]
For $n$ large enough so that $2^{\left(\frac{\alpha}{h}-1\right)j_n'}<\delta
$ and $2^{-\varepsilon j_n'}
<
\delta\,2^{-\left(\rho-1+\frac{\alpha}{h}\right)}$,
we obtain
\begin{align*}
    \#\big\{\lambda'\in\Lambda_{j_n'}:c_{\lambda'}=2^{-\alpha j_n'}\big\}
    & \geq \big(2^{\left((\rho-1)\frac{h}{\alpha}+1\right)\frac{\alpha}{h}(j_n'-1)}-1\big)\big(2^{(\alpha p+1)j_n'-(hp+1)\frac{\alpha}{h}\,j_n'}-1\big) \\
    & \geq 2^{-\big(\rho-1+\frac{\alpha}{h}\big)}2^{\rho j_n'}-2^{-\big(\rho-1+\frac{\alpha}{h}\big)\rho j_n'}2^{\big(\frac{\alpha}{h}-1\big)j_n'}-2^{(\rho-\varepsilon)j_n'} \\
    & \geq 2^{-\big(\rho-1+\frac{\alpha}{h}\big)}2^{\rho j_n'}(1-2\delta),
\end{align*}
On the other hand, for every $j\in\N$,
\[
\#\left\{\lambda'\in\Lambda_j:
c_{\lambda'}=2^{-\alpha j}\right\}
\leq
2^{(hp+1)\frac{\alpha}{h}}\,2^{\rho j},
\]
with the left-hand side equal to zero whenever
$j\notin\{j_n':n\in\N\}$.
It follows that $\rho_{\vec c}(\alpha)=\rho$. 

We now verify Property~(3). For every $j\in\N$,
\[
2^{-j}\sum_{\lambda\in\Lambda_j}|c_\lambda|^p
\leq
2^{-j}
2^{(hp+1)\frac{\alpha}{h}}
2^{\rho j}
2^{-\alpha pj}
<
2^{(hp+1)\frac{\alpha}{h}}2^{-\varepsilon j},
\]
where the last inequality follows from
$\rho<\alpha p+1-\varepsilon$. Hence $\eta_f(p)\geq\varepsilon>0$.

It remains to prove Property~(4). If $n$ is large enough to satisfy $1-2^{-\varepsilon pj_n}\geq 2^{-\varepsilon pj_n}$, then for every
$\lambda\in\widetilde{\Lambda}_{j_n}$,
\begin{align*}
e_\lambda^{(p)}
\geq
\Bigg(\sum_{\lambda'\in\Lambda_{j_n'}\,:\,\lambda'\subseteq\lambda}
\modu{c_{\lambda'}}^p2^{-(j_n'-j_n)}\Bigg)^{\frac1p}
&\geq
\left(\left\lfloor 2^{(\alpha p+1)j_n'-(hp+1)j_n}\right\rfloor
2^{-\alpha pj_n'}2^{-(j_n'-j_n)}\right)^{\frac1p}\\
&\geq
\left(
2^{-hpj_n}
\left(1-2^{\left(1-\frac{h}{\alpha}\right)j_n}\right)
\right)^{\frac1p}\\
&\geq
\left(
2^{-hpj_n}\left(1-2^{-\varepsilon pj_n}\right)
\right)^{\frac1p}
\geq
2^{-(h+\varepsilon)j_n}.
\end{align*}
Consequently,
\[
h_f^{(p)}(x)\leq h
\qquad\text{for every}\qquad
x\in
E\coloneqq
\limsup_{n\to\infty}
\bigcup_{\lambda\in\widetilde{\Lambda}_{j_n}}\lambda.
\]
We shall now prove that
\[
\dim_{\mathcal H}(E)
\geq
(\rho-1)\frac{h}{\alpha}+1.
\]
Recall that the intervals of $\widetilde{\Lambda}_{j_n}$ were chosen as
uniformly as possible. Hence the distance between two consecutive such
intervals, as well as the distances from the endpoints $0$ and $1$, is at
most
\[
\Big\lceil
\frac{2^{j_n}-N_n}{N_n-1}
\Big\rceil 2^{-j_n}.
\]
Consequently, the distance between the centers of two consecutive
intervals of $\widetilde{\Lambda}_{j_n}$ is at most
\[
\Big\lceil\frac{2^{j_n}-N_n}{N_n-1}\Big\rceil 2^{-j_n}+2^{-j_n} 
\leq 2\Big(\frac{2^{-j_n}}{2}\Big)^{(\rho-1)\frac{h}{\alpha}+1}, 
\]
for all $n$ large enough. Hence, for $n$ large enough, any two consecutive dilated intervals $\lambda^{(\rho-1)\frac{h}{\alpha}+1}$ ($\lambda\in\widetilde{\Lambda}_{j_n}$) overlap. Consequently, their union covers $\ccint{0,1}$ except possibly
for two intervals adjacent to the endpoints, and therefore 
\[
\leb\Big(\bigcup_{\lambda\in\widetilde{\Lambda}_{j_n}}\lambda^{(\rho-1)\frac{h}{\alpha}+1}\Big) \geq 1-2\Big(\frac{2^{-j_n}}{2}\Big)^{(\rho-1)\frac{h}{\alpha}+1},
\]
hence
\[
\leb\Big(\bigcup_{n\geq N}\bigcup_{\lambda\in\widetilde{\Lambda}_{j_n}}\lambda^{(\rho-1)\frac{h}{\alpha}+1}\Big) \geq 1.
\]
The mass transference principle therefore  provides a gauge function $\xi$ such that
\[
\lim_{r\to0^+}\frac{\log\xi(r)}{\log r}
=
(\rho-1)\frac{h}{\alpha}+1
\]
and $\mathcal H^\xi(E)>0$. Consequently,
$
\dim_{\mathcal H}(E)
\geq
(\rho-1)\frac{h}{\alpha}+1.
$
Since the sets of points with $p$-exponent strictly smaller than $h$
have Hausdorff dimension strictly smaller than $(\rho-1)\frac{h}{\alpha}+1$, a standard argument
gives
\[
\mathscr D_f^{(p)}(h)\ge (\rho-1)\frac{h}{\alpha}+1.
\]

\bigskip

\noindent \textbf{Acknowledgements.} This work was supported by an FNRS grant awarded to T. Lambert. This project was also partially supported by the Tournesol program, a Partenariat Hubert Curien (PHC), and by the CNRS International Research Project  FRALEA : Fractalité, Multifractalité et Aléa. C. Esser also acknowledges support from the Fonds Thelam, managed by the King Baudouin Foundation (Belgium). The authors thank \'E. Daviaud for fruitful discussions that greatly contributed to the development of this study.

\bibliographystyle{plain}
\bibliography{pspectrum.bib}

\end{document}